\documentclass[reqno]{amsart}

\input xy
\xyoption{all}
\usepackage{accents}
\usepackage{epsfig}
\usepackage{color}
\usepackage{amsthm}
\usepackage{amssymb}
\usepackage{amsmath}
\usepackage{amscd}
\usepackage{amsopn}
\usepackage{epstopdf}
\usepackage{mathrsfs}
\usepackage{subfigure}
\usepackage{float}
 \usepackage[color=yellow]{todonotes}

\usepackage{graphicx}

\usepackage{xspace}

\usepackage{url}
\usepackage{enumitem, hyperref}\hypersetup{colorlinks}

\usepackage{color} 

\definecolor{darkred}{rgb}{1,0,0} 
\definecolor{darkgreen}{rgb}{0,0.8,0}
\definecolor{darkblue}{rgb}{0,0,1}

\hypersetup{colorlinks,
linkcolor=darkblue,
filecolor=darkgreen,
urlcolor=darkred,
citecolor=darkgreen}

\makeatletter
\def\reflb#1#2{\begingroup
    #2%
    \def\@currentlabel{#2}%
    \phantomsection\label{#1}\endgroup
}
\makeatother

\newtheorem{theorem}{Theorem}[section]

\newtheorem{corollary}[theorem]{Corollary}

\newtheorem{definition}[theorem]{Definition}
\newtheorem{example}[theorem]{Example}
\newtheorem{remark}[theorem]{Remark}

\newtheorem{lemma}[theorem]{Lemma}

\newtheorem{proposition}[theorem]{Proposition}

\numberwithin{equation}{section}

\numberwithin{Theorem}{section}

\theoremstyle{definition}

\theoremstyle{remark}

\def\m{\langle\cdot,\cdot\rangle}

\def    \12    {{\frac{1}{2}}}

\begin{document}


\setlength{\smallskipamount}{6pt}
\setlength{\medskipamount}{10pt}
\setlength{\bigskipamount}{16pt}






\title[Symplectic deformations of Floer homology and non-contractible periodic orbits in twisted disc bundles]{Symplectic deformations of Floer homology and non-contractible periodic orbits in twisted disc bundles}

\author[Wenmin Gong]{Wenmin Gong}

\address{School of Mathematical Sciences, Beijing Normal University,
    Beijing, 100875, China} \email{wmgong@bnu.edu.cn}

\subjclass[2010]{53D40; 37J45; 70H12}

\keywords{Twisted cotangent bundles, Floer homology, Symplectic deformations, Non-contractible periodic orbits, Almost existence theorem}


\thanks{The work is partially supported by  NSFC 11701313, the Fundamental Research Funds for the Central Universities 2018NTST18 and by China Postdoctoral Science Foundation grant 2017M620725.}


\begin{abstract} In this paper we establish the existence of periodic orbits belonging to any $\sigma$-atoroidal free homotopy class for Hamiltonian systems in the twisted disc bundle, provided that the compactly supported time-dependent Hamiltonian function is sufficiently large over the zero section and the magnitude of the weakly exact $2$-form $\sigma$ admitting a primitive with at most linear growth on the universal cover is sufficiently small. The proof relies on showing the invariance of Floer homology under symplectic deformations and on the computation of Floer homology for the cotangent bundle endowed with its canonical symplectic form.

As a consequence, we also prove that, for any nontrivial atoroidal free homotopy class and any positive finite interval, if the magnitude of a magnetic field admitting a primitive with at most linear growth on the universal cover is sufficiently small, the twisted geodesic flow associated to the magnetic field has a periodic orbit on almost every energy level in the given interval whose projection to the underlying manifold represents the given free homotopy class. This application is carried out by showing the finiteness of the restricted Biran-Polterovich-Salamon capacity.
\end{abstract}

\maketitle


\tableofcontents


\section{Introduction}\label{sec:1}
\setcounter{equation}{0}
Let $(M,g)$ be a closed connected Riemannian manifold with cotangent bundle $\pi:T^*M\to M$. Denote by $\omega_0=-d\lambda$ the canonical symplectic form on $T^*M$, where $\lambda=pdq$ is the Liouville $1$-form in  local coordinates $(q,p)$ of $T^*M$. Let $(\tilde{M},\tilde{g})$ be the universal cover of $(M,g)$. For a closed $2$-form $\sigma\in \Omega^2(M)$ we denote by $\tilde{\sigma}$ the lift of $\sigma$ to $\tilde{M}$, and say that $\sigma$ is \emph{weakly exact} if $\tilde{\sigma}$ is exact. Let $D_RT^*M$ denote the open disc cotangent bundle of finite radius $R$ with respect to the metric $g$.  Denote by $\Omega_w^2(M)$ the set of all weakly exact $2$-forms on $M$. For each $\sigma\in \Omega_w^2(M)$ we denote
$$\omega_\sigma:=\omega_0+\pi^*\sigma$$
the \emph{twisted symplectic form}, and call the symplectic manifold $(T^*M,\omega_\sigma)$ the \emph{twisted cotangent bundle} and $(D_RT^*M,\omega_\sigma)$ the \emph{twisted disc bundle}.

For a smooth function $H:S^1\times T^*M\to \mathbb{R}$ we are interested in the existence of periodic solutions of the associated Hamiltonian system
$$\dot{x}(t)=X_{H,\sigma}(t,x(t))\quad\forall t\in S^1,$$
where the Hamiltonian vector field $X_{H,\sigma}$ on $T^*M$ is determined by $dH_t=\iota(X_{H,\sigma})\omega_{\sigma}$. The existence problem of contractible Hamiltonian periodic solutions on the twisted cotangent bundle has been studied in a number of papers; see, e.g., \cite{CGK,FS,GG0,GG1,Lu1,Lu2,Lu3,Us}. On the other hand, non-contractible periodic orbits of the Hamiltonian systems on the ordinary cotangent bundle have been previously investigated in several papers; see, e.g., \cite{BPS,GL,Ir,We0,Xu}.

The aim of the present paper is to study certain existence results for
$1$-periodic orbits in given free homotopy classes of loops for compactly supported Hamiltonians on twisted cotangent bundles. The proof relies on the machinery of Floer homology for non-contractible periodic orbits.
In the course of the last two decades, this version of Floer homology has been studied and used among others
in 2003 by Biran, Polterovich and Salamon~\cite{BPS} on $T^*M$ for
$M=\mathbb{T}^n$ or $M$ closed and
negatively curved, in 2006 by Weber~\cite{We0} on $T^*M$ for all closed Riemannian manifolds $M$, in 2013 by G\"{u}rel~\cite{Gu} on closed symplectic manifolds for simple non-contractible periodic orbits of Hamiltonian diffeomorphisms with arbitrarily large periods  and in 2016 by Ginzburg and G\"{u}rel~\cite{GG1} on closed symplectic manifolds for infinitely many non-contractible periodic orbits . For further references
concerning the existence of non-contractible orbits we refer to \cite{Ba,Ci,GG2,GL,KO,Ni,Xu}.

The main difficulty in applying this technique to detect non-contractible periodic orbits of the Hamiltonian flow of $H$ is to compute the filtered Floer homology ${\rm HF}_\alpha^{(a,b)}(H)$. It is well known that the method adopted by
Biran, Polterovich and Salamon in~\cite{BPS} is applied successfully to many other situations~\cite{CGK,GG0,GG1,Us,Xu,KO}. The basic strategy is to examine the commutative diagram

\begin{equation}\notag\label{e:ch}
\xymatrix{{\rm HF}_\alpha^{(a,b)}(H_0)
\ar[rr]^{\Psi_{H_1H_0}}\ar[dr]_{\Psi_{HH_0}}& & {\rm HF}_\alpha^{(a,b)}(H_1)\\ & {\rm HF}_\alpha^{(a,b)}(H)\ar[ur]_{\Psi_{H_1H}} & }
\end{equation}
where $H_0\leq H\leq H_1$ and $H_0$ and $H_1$ are time-independent Hamiltonians with ${\rm HF}_\alpha^{(a,b)}(H_0)$ and ${\rm HF}_\alpha^{(a,b)}(H_1)$
being not hard to compute. However, in our cases, one can not directly use the method mentioned above which is based on Pozniak's Theorem~\cite{Poz} since the sets of critical points of the symplectic action functionals of the squeezing functions $H_0$ and $H_1$, in general, are not Morse-Bott in the sense of ~\cite{BPS}. Instead, we interpret the Floer homology in the twisted cotangent bundle as a ``small perturbation'' of the Floer homology in the ordinary cotangent bundle and show the isomorphism between them.  Of course we shall have to make suitable assumptions about the twisted symplectic forms.  For a nice introduction to the theories of deformations of Floer homology on a closed symplectic manifold under symplectic perturbations we refer to the papers by Ritter~\cite{Ri}, Zhang~\cite{Zh} and Bae and Frauenfelder~\cite{BF}. It is worth to mention that the monotone homomorphisms (see Section~\ref{subsec:MH}) are preserved under the  isomorphisms which we obtain in this paper. This observation, together with the computation of Floer homology given by Weber~\cite{We0}, helps us to circumvent the difficulty of computing the Floer homology on the twisted disc bundle directly. Our results can be seen as a twisted version of Weber's results~\cite{We0} which, in particular, recover a result of Niche~\cite{Ni}.

\subsection{Main results}\label{sec:1.1}
\setcounter{equation}{0}

\begin{definition}
{ \rm
A $1$-form $\tau\in \Omega^1(\tilde{M})$ is said to be a primitive of $\sigma$ with \emph{at most linear growth} on the universal cover $\tilde{M}$ if $d\tau=\tilde{\sigma}$ and, for any $x\in \tilde{M}$, there exists a constant $\varrho_x>0$ such that for all $s\geq 0$,
}
\begin{equation}\label{e:LG}
\sup\limits_{z\in B_x(s)}
\|\tau_z\|_{\tilde{g}}\leq \varrho_x(s+1).
\end{equation}
\end{definition}
\noindent Here $B_x(s)$ denotes the geodesic ball in $\tilde{M}$ of radius $s>0$ centered at $x$. Since $M$ is compact, the above definition is independent of the choice of the metric $g$ on $M$. Put

\begin{equation}
u_{\sigma,g,x}(s)=\inf\limits_{\tau\in \mathcal{F}_\sigma}\sup\limits_{z\in B_x(s)}
\|\tau_z\|_{\tilde{g}}\quad \forall s>0,
\end{equation}
where $\mathcal{F}_\sigma=\{\tau\in\Omega^1(\tilde{M})|d\tau=\tilde{\sigma}\}$.
We call $u_{\sigma}=u_{\sigma,g,x}:\mathbb{R}_+\to\mathbb{R}_+$ the \emph{cofilling function} of $\sigma\in\Omega_w^2(M)$; see \cite{BF,Gr,Po}. If $g^\prime$ is another Riemannian metric on $M$ and $x^\prime\in\tilde{M}$ is a different base point then it is easy to check that $L^{-1}u_{\sigma,g,x}\leq u_{\sigma,g^\prime,x^\prime}\leq L u_{\sigma,g,x}$ for some constant $L>0$.

We define the set $\mathcal{P}(M)$ as
\begin{equation}\label{e:CW2F}
\mathcal{P}(M):=\{\sigma\in\Omega_w^2(M)|\hbox{ there exists a constant}\;\epsilon>0\;\hbox{such that}\; u_\sigma(s)\leq \epsilon(s+1)\}.
\end{equation}
\noindent Notice that
 each $\sigma \in \mathcal{P}(M)$, by definition, admits a primitive with at most linear growth on the universal cover $\tilde{M}$, and hence $\sigma|_{\pi_2(M)}=0$. So $\omega_\sigma$ is \emph{symplectically aspherical}, which means that the symplectic form $\omega_\sigma$ vanishes on $\pi_2(T^*M)$.

\begin{example}\label{ex:3exa}
{\rm
Let us list some manifolds on which there exists a weakly exact $2$-form admitting a primitive with at most linear growth on the universal cover; see, e.g., Gromov \cite{Gr}, Polterovich \cite{Po}, Sikorav \cite{Si} and references therein.

\begin{enumerate}
  \item If a Riemannian manifold $M$ is closed and of non-positive curvature then every closed 2-form $\sigma$ on $M$ admits a primitive with at most linear growth on $\tilde{M}$. Moreover, whenever $M$ admits a metric of negative curvature, $\tilde{\sigma}$ admits a bounded primitive, that is, there exists a $1$-form $\tau\in \Omega^1 (\tilde{M})$ such that $d\tau=\tilde{\sigma}$ and
      \begin{equation}\label{e:bp}
      \sup\limits_{x\in \tilde{M}}\|\tau_x\|<\infty.
      \end{equation}
      In particular, from (1) we have the following examples:
  \item Assume that $M$ is a closed oriented Riemannian surface with infinite fundamental group. Then the volume form on $M$ admits a primitive with at most linear growth on $\tilde{M}$.
  \item For the standard sympletic torus $(\mathbb{T}^{2n},\sum_{i=1}^{2n}dx_i\wedge dy_i)$ it holds that $\sum_{i=1}^{2n}dx_i$ $\wedge dy_i\in \mathcal{P}(\mathbb{T}^{2n})$. In fact, for $n$-dimensional tori $\mathbb{T}^n$ any closed non-exact 2-form $\sigma$ admits a primitive with at most linear growth on the universal cover $\mathbb{R}^n$ but does not satisfy~(\ref{e:bp}); see \cite{FMP0}.

\end{enumerate}
}
\end{example}

\begin{remark}\label{rem:LGC}
{\rm
When $(M,g)$ is a closed Riemannian manifold of negative curvature, the constant $\epsilon=\epsilon(\sigma)$ in $\mathcal{P}(M)$ can be chosen to converge to zero as $|\sigma|_g\to0$ Hereafter we denote $|\sigma|_g\overset{\rm def}{=}\sup_{x\in M}\|\sigma(x)\|_g$. Indeed, there exists a universal constant $\rho>0$ such that every closed $2$-form $\beta\in\Omega^2(M)$ with $|\beta|_g\leq 1$ satisfies $u_\beta(s)\leq \rho,\;\forall s\in[0,\infty)$; see Gromov~\cite[$5. {\rm B}_5$]{Gr}. Then for any closed 2-form $\sigma\in\Omega^2(M)$ the rescaling form $\hat{\sigma}=\sigma/|\sigma|_g$ satisfies $|\hat{\sigma}|_g=1$, $\hat{\sigma}\in\mathcal{P}(M)$ and $u_{\hat{\sigma}}(s)\leq \rho$. This implies that $u_{\sigma}(s)\leq \epsilon(\sigma):=\rho|\sigma|_g$ and $\epsilon(\sigma)\to 0$ as $|\sigma|_g\to 0$.
}
\end{remark}

Let $S^1=\mathbb{R}/\mathbb{Z}$, and we denote the free loop space of $M$ by $\mathfrak{L} M:=C^\infty(S^1,M)$. Given a free homotopy class $\alpha\in[S^1,M]$ we define
$$\mathfrak{L}_\alpha M:=\{\gamma\in \mathfrak{L} M\big|[\gamma]=\alpha \}.$$
We use the symbol $\Lambda_\alpha$ for the set  of lengths of all periodic geodesics in $M$ with respect to $g$ which represent $\alpha$. It is a closed and nowhere dense subset of $\mathbb{R}$ (see~\cite[Lemma 3.3]{We0}), and hence there exists a periodic geodesic $q$ such that
$$l_\alpha:=\int_{S^1}\|\dot{q}(t)\|_gdt=\inf\Lambda_\alpha.$$

\begin{definition}
{\rm Let $\sigma$ be a two-form on $M$. A class $\alpha\in [S^1,M]$ is said to be a \emph{$\sigma$-atoroidal class} if any map $f:S^1\to \mathfrak{L}_\alpha M$, thought of as a map $f:\mathbb{T}^2\to M$, satisfies $\int_{\mathbb{T}^2}f^*\sigma=0$.
}
\end{definition}
\noindent Note that if $\sigma$ is weakly exact, then the class $0$ of nullhomotopic loops is atoroidal, since both statements are equivalent to the statement that $\sigma\big|_{\pi_2(M)}=0$.
\begin{remark}\label{rem:NC}
{\rm If $M$ admits a metric of negative curvature, then every smooth map $f:\mathbb{T}^2\to M$ induces the zero map
$f^*:H^2_{dR}(M,\mathbb{R})\to H_{dR}^2(\mathbb{T}^2,\mathbb{R})$; see~\cite[Lemma~2.3]{Me} or \cite{Ni}. Therefore, for any closed $2$-form $\sigma\in \Omega^2 (M)$, every free homotopy class $\alpha\in [S^1,M]$ is a $\sigma$-atoroidal class. Another example is the $3$-dimensional torus $\mathbb{T}^3=\{(x,y,z)|x,y,z\in\mathbb{R}/\mathbb{Z}\}$ with $\sigma=dx\wedge dy$. It is easy to check that the homotopy class $\alpha=(0,0,n)$ for any integer $n$ is $\sigma$-atoroidal, see also~\cite{FMP0}.
}
\end{remark}

Throughout this paper, for the sake of brevity, we always assume that $\sigma \in \mathcal{P}(M)$ and $\alpha\in[S^1,M]$ is a $\sigma$-atoroidal class without any special statement. Given a homotopy class $\alpha$ of free loops in $M$, denote by $\mathscr{P}_{\alpha}(H,\sigma)$ the set of periodic orbits of $X_{H,\sigma}$:
$$\mathscr{P}_{\alpha}(H,\sigma)=\{x\in C^\infty(S^1,T^*M)\big|\dot{x}(t)=X_{H,\sigma}(t,x(t))\quad\forall t\in S^1,\;[\pi(x)]=\alpha \}.$$
Our first main result establishes the following existence of non-contractible periodic orbits for compactly supported Hamiltonians on twisted cotangent bundles.

\begin{theorem}\label{thm:1}
Let $(M,g)$ be a closed connected Riemannian manifold, and let $D_RT^*M$ be the open disc cotangent bundle of finite radius $R$ with respect to the metric $g$. Assume that $\sigma \in \mathcal{P}(M)$ and $\alpha\in[S^1,M]$ is any $\sigma$-atoroidal class. For every compactly supported Hamiltonian $H\in C^\infty(S^1\times D_RT^*M)$ with
$$\sup\limits_{S^1\times M}H< -Rl_{\alpha}$$
there exists $\delta_0(H,g,\sigma,\alpha)>0$ such that if $|\delta|<\delta_0(H,g,\sigma,\alpha)$ then $\mathscr{P}_{\alpha}(H,\delta\sigma)\neq \emptyset$.
\end{theorem}

Moreover, if $M$ is a closed and negatively curved manifold we have the following.
\begin{theorem}[Niche~\cite{Ni}]\label{thm:2}
Assume that $(M,g)$ is a closed Riemannian manifold of negative curvature, and that $\alpha\in[S^1,M]$ is any free homotopy class. Denote by $\sigma$ a closed $2$-form on $M$. Let $H\in C^\infty(S^1\times D_RT^*M)$ be a compactly supported function.
Then there exist some positive constants $C=C(\alpha)$ and $\delta_0=\delta_0(H,g,\alpha)$ such that if
$$\inf\limits_{S^1\times M}H>C$$ then $\mathscr{P}_{-\alpha}(H,\sigma)\neq \emptyset$ whenever $|\sigma|_g<\delta_0$.
\end{theorem}

\begin{remark}
{\rm In \cite{Ni} the author used Pozniak's theorem to compute Floer homology of the squeezing functions $K^{\pm}$ for a twisted cotangent bundle $T^*M$ of a negatively curved manifold $M$, see~\cite[Proposition~3]{Ni}. He proved that the set $\mathcal{P}_\omega(\rho)$ of $1$-periodic orbits of Hamiltonian $K^+$ with respect to the twisted symplectic form $\omega$ is homeomorphic to $S^1$, see lines 20-25, page 627 in \cite{Ni}. However, it seems very difficult for us to verify the condition that $\mathcal{P}_\omega(\rho)$ is a Morse-Bott manifold, meaning that $C_0=\{x(0):x\in \mathcal{P}_\omega(\rho)\}$ is a compact submanifold $M$ and $T_{x_0}C_0={\rm Ker}(d\phi_1(x_0)-Id)$ for every $x_0\in C_0$, where $\phi_1$ is the time one flow of $K^+$ with respect to $\omega$ (not to $\omega_0$!). This condition is the key to make use of Pozniak's theorem. For this reason, we devise a different method to show Niche's result, which is a partial motivation of our work.  
}
\end{remark}

Theorem~\ref{thm:1} (resp.Theorem~\ref{thm:2}) is a soft consequence of the invariance of Floer homology under symplectic deformations (see Theorem~\ref{thm:Invariance} (resp. Theorem~\ref{thm:Inv})).
As a result, to obtain periodic orbits one needs to bound the magnitude of the weak exact $2$-form in terms of $H$. To get rid of the dependence on $H$, we introduce a class of Hamiltonian functions compactly supported in $D_RT^*M$, and show the finiteness of a symplectic capacity defined by it. To state this result, we need to put more restrictive assumptions on $H$, and the capacity defined here is slightly different from the \emph{Biran-Polterovich-Salamon (BPS) capacity} (for the original definition see~\cite{BPS,We0}). These additional constraints on the class of Hamiltonians to define the capacity are natural because we will apply the Theorem~\ref{thm:Invariance} (resp. Theorem~\ref{thm:Inv}) to two \textit{fixed} Hamiltonians sandwiching at the same time the whole class of functions to estimate the capacity, see~Figure~\ref{fig:9}.

Let us denote $\mathscr{P}_{\alpha}(H,\sigma;\tau)$ the set of $\tau$-periodic orbits of $X_{H,\sigma}$ representing $\alpha\in[S^1,M]$ (identifying $S^1$  with $\mathbb{R}/(\tau\mathbb{Z})$). We say a periodic orbit $x\in \mathscr{P}_{\alpha}(H,\sigma;\tau)$ \emph{fast} if $0<\tau\leq 1$; otherwise, we say it \emph{slow}.
Let $W$ be an open subset of $T^*M$ containing $M$, and let $U$ be an open subset of $T^*M$ with compact closure $\bar{U}\subset W$.
Given a number $A>0$, denote by $\mathscr{H}(W,U,A)$ the class of smooth functions $H:W\to \mathbb{R}$ such that
\begin{itemize}
  \item[(H0)] $H(x)\leq 0$, for all $x\in W$;
  \item[(H1)] $H$ is compactly supported in $U$; and
  \item[(H2)] $\inf_{W}H>-A$.
\end{itemize}
Let $M\subset V\subset U$, where $V$ is an open subset of $T^*M$. For $c\in(0,A]$, denote
$$\mathscr{H}_c(W,U,V,A):=\big\{H\in\mathscr{H}(W,U,A) \big| \sup_{V}H\leq-c \big\}.
$$
Then the \emph{restricted BPS capacity $\hat{c}_{\rm BPS}$} is defined as
$$\hat{c}_{\rm BPS}(W,U,V,A;\sigma,\alpha)=\inf\big\{c>0\big|\forall H\in \mathscr{H}_c(W,U,V,A),\;\hbox{there is a fast periodic of $H$}\big\}.$$
Here we use the convention that $\inf\emptyset=\infty$.
\begin{remark}
{\rm It is easy to check that if $\hat{c}_{\rm BPS}(W,U,V,A;\sigma,\alpha)<\infty$, then every $H\in\mathscr{H}(W,U,A)$ satisfying $\sup_{V}H\leq-\hat{c}_{\rm BPS}(W,U,V,A;\sigma,\alpha)$ has at least a fast periodic orbit (with respect to the twisted symplectic form $\omega_{\sigma}=\omega_0+\pi^*\sigma$) whose projection to $M$ represents $\alpha$.
}
\end{remark}
In what follows, let us denote $U_R:=D_RT^*M$,
let $V$ be an open subset of $U_{R-\rho}$ containing $M$ with $0<\rho<R$.

\begin{theorem}\label{thm:3}
Let $(M,g)$ be a closed connected Riemannian manifold. Denote $V,A,\rho$ as above.  Assume that $\sigma \in \mathcal{P}(M)$ and $\alpha\in[S^1,M]$ is any $\sigma$-atoroidal class. If $A>Rl_\alpha$, then for any sufficiently small number $\epsilon>0$ there exists a constant $\delta_0=\delta_0(g,\sigma,\alpha, V,A,\rho,\epsilon)>0$ such that for every $\delta\in(-\delta_0,\delta_0)$ it holds that
$$\hat{c}_{\rm BPS}(U_R,U_{R-\rho},V,A;\delta\sigma,\alpha)\leq Rl_\alpha+\epsilon.$$

\end{theorem}

\begin{theorem}\label{thm:4}
Let $(M,g)$ be a closed Riemannian manifold of negative curvature, and let $\alpha\in[S^1,M]$ be a free homotopy class. Denote $V,A,\rho$ as above, and denote by $\sigma$ any closed $2$-form on $M$. If $A>Rl_\alpha$, then for any sufficiently small number $\epsilon>0$ there exists a constant $\delta_0=\delta_0(g,\alpha, V,A,\rho,\epsilon)>0$ such that if $|\sigma|_g<\delta_0$ then we have
$$\hat{c}_{\rm BPS}(U_R,U_{R-\rho},V,A;\sigma,\alpha)\leq Rl_\alpha+\epsilon.$$
\end{theorem}

Fix $\rho_0\in (0,R-\rho)$ and take $V=U_{\rho_0}$. By the definition of $\hat{c}_{\rm BPS}$ and a scaling argument like in the proof of Proposition~\ref{pro:nonexistence}, one can obtain 
$$\delta\cdot \hat{c}_{\rm BPS}(U_R,U_{R-\rho},U_{
	\rho_0},A;\sigma,\alpha)=\hat{c}_{\rm BPS}(U_{\delta R},U_{\delta(R-\rho)},U_{
	\delta\rho_0},\delta A;\delta\sigma,\alpha)$$
for any positive number $\delta$. Combining this with Theorem~\ref{thm:3} and Theorem~\ref{thm:4} we have the following two theorems.

\begin{theorem}\label{thm:3'}
	Let $(M,g)$ be a closed connected Riemannian manifold. Denote $A,\rho_0, \rho$ as above.  Assume that $\sigma \in \mathcal{P}(M)$ and $\alpha\in[S^1,M]$ is any $\sigma$-atoroidal class. If $A>Rl_\alpha$, then for any sufficiently small number $\epsilon>0$ there exists a constant $\delta_0=\delta_0(g,\sigma,\alpha, A,\rho,\rho_0,\epsilon)>0$ such that for any $\lambda\in (0,+\infty)$ it holds that 
	$$\hat{c}_{\rm BPS}(U_{\lambda R},U_{\lambda(R-\rho)},U_{\lambda\rho_0},\lambda A;\delta_0\lambda\sigma,\alpha)\leq\lambda(Rl_\alpha+\epsilon).$$

\end{theorem}

\begin{theorem}\label{thm:4'}
	Let $(M,g)$ be a closed Riemannian manifold of negative curvature, and let $\alpha\in[S^1,M]$ be a free homotopy class. Denote $A,\rho,\rho_0$ as above, and denote by $\sigma$ any closed $2$-form on $M$. If $A>Rl_\alpha$, then for any sufficiently small number $\epsilon>0$ there exists a constant $\delta_0=\delta_0(g,\alpha, A,\rho,\rho_0,\epsilon)>0$ such that 
$$\hat{c}_{\rm BPS}(U_{\delta_0R|\sigma|_g},U_{\delta_0(R-\rho)|\sigma|_g},U_{\delta_0\rho_0|\sigma|_g},\delta_0A|\sigma|_g;\sigma,\alpha)\leq\delta_0(Rl_\alpha+\epsilon)|\sigma|_g.$$
\end{theorem}

\begin{remark}
{\rm
In the definition of the restricted BPS capacity, one can also allow $H$ to be time-dependent. In this case, we require that $H_t,t\in[0,1]$ to be periodic in time, conditions (H0)-(H2) hold uniformly for $t\in S^1$, and that $\sup_{S^1\times V}H\leq -c$ in the definition of $\mathscr{H}_c(W,U,V,A)$. Then one can show that Theorem~\ref{thm:3} -- Theorem~\ref{thm:4'} still hold.
}
\end{remark}

Like the Hofer-Zehnder capacity, the finiteness of restricted BPS capacity implies almost existence of periodic orbits.
Consider a proper\footnotemark[1]\footnotetext[1]{A map is called proper if preimages of compact sets are compact.} smooth function
$$H:T^*M\to \mathbb{R}$$
which is bounded from below. Then $\{H\leq s\}$ is compact, and hence $\{H<s\}$ is contained in $D_rT^*M$ for some sufficiently large radius $r=r(s)$. Suppose that $M\subset \bar{V}\subset\{H< d\}$, where $V$ is an open neighborhood of $M$ in $T^*M$. Given $\rho>0$ and $T>d$, let $A>(r(T+\rho)+\rho)l_\alpha$. Then
the finiteness of $\hat{c}_{\rm BPS}(U_{r(T+\rho)+\rho},U_{r(T+\rho)},V,A;\delta\sigma,\alpha)$ implies the following.

\begin{theorem}[Almost existence theorem]\label{thm:AET1}
Let $(M,g)$ be a closed connected Riemannian manifold. Assume $\sigma \in \mathcal{P}(M)$ and $\alpha\in[S^1,M]$ is a non-trivial $\sigma$-atoroidal class. Let $H:T^*M\to \mathbb{R}$ be a proper smooth function bounded from below, and denote $d,V,A,T,\rho$ as above. Then there exists a constant $\delta_0=\delta_0(g,\sigma,\alpha, V, T, A,$ $\rho)>0$ such that
if $\delta\in(-\delta_0,\delta_0)$ then for almost all $s\in[d,T]$ the level set $\{H=s\}$ carries a periodic Hamiltonian orbit with respect to $\omega_{\delta\sigma}=\omega_0+\delta\pi^*\sigma$ whose projection to $M$ represents $\alpha$.

\end{theorem}

 Similarly, for manifolds of negative curvature the finiteness of  $\hat{c}_{\rm BPS}(U_{r(T+\rho)+\rho},$ $U_{r(T+\rho)},V,A;\sigma,\alpha)$ implies the following.

\begin{theorem}[Almost existence theorem]\label{thm:AET2}
Let $(M,g)$ be a closed Riemannian manifold of negative curvature, and let $\alpha\in[S^1,M]$ be any non-trivial free homotopy class. Let $H:T^*M\to \mathbb{R}$ be a proper smooth function bounded from below. Denote by $\sigma$ any closed $2$-form on $M$, and denote $d,V,A,T,\rho$ as above. Then there exists a constant $\delta_0=\delta_0(g,\alpha, V, T, A,\rho)>0$ such that
if $|\sigma|<\delta_0$ then for almost every $s\in[d,T]$ the level set $\{H=s\}$ carries a periodic Hamiltonian orbit with respect to $\omega_{\sigma}=\omega_0+\pi^*\sigma$ whose projection to $M$ represents $\alpha$.

\end{theorem}

\subsection{Applications}\label{sec:1.2}
\setcounter{equation}{0}
Recall that the Hamiltonian function $H_g(q,p)=\|p\|^2_g/2$ is said to be the \emph{standard kinetic energy} on $T^*M$, the Hamiltonian flow of $X_{H_g,\sigma}$ on $T^*M$, called a \emph{twisted geodesic flow}, describes the motion of a charge on $M$ in the \emph{magnetic field} $\sigma$. For related results on the existence of periodic orbits of twisted geodesic flows, we refer to the papers by Ginzburg and G\"{u}rel~\cite{GG0,GG1} using the Floer theory for contractible periodic orbits, by Lu~\cite{Lu3} applying pseudo symplectic capacities for magnetic fields given by symplectic forms and by Frauenfelder and Schlenk~\cite{FS} using the spectral metric when the two-form $\sigma$ is exact. Recently, Asselle and Benedetti~\cite{AB} showed that for almost all energy levels above the maximum critical value of an autonomous Tonelli Hamiltonian  there exists a periodic magnetic geodesic based on variational methods. By the same method, Abbondandolo,  Macarini,  Mazzucchelli and  Paternain~\cite{AMMP} proved the existence of infinitely many periodic orbits of exact magnetic flows on surfaces for almost every subcritical energy level. It is worth to mention that Ginzburg~\cite{Gi1} also gave a counterexample and showed the nonexistence of closed trajectories of $H_g$; see~\cite{Gi0} for a more comprehensive survey. Numerous results about periodic orbits of the Hamitonian flow in twisted cotangent bundles can be found in \cite{FMP0,FMP1,GK,Lu1,Lu2,Us,Xu}. For $H_g$ the function $r(s)$ can be taken as $\sqrt{2s}$. So by using Theorem~\ref{thm:AET1} we obtain the following.

\begin{corollary}\label{cor:AET1}
Let $(M,g)$ be a closed connected Riemannian manifold. Assume $\sigma \in \mathcal{P}(M)$ and $\alpha\in[S^1,M]$ is a non-trivial $\sigma$-atoroidal class. Then for $0<d<T$,
there exists a positive constant $\delta_0=\delta_0(g,\sigma,\alpha,d,T)$ such that if $\delta\in(-\delta_0,\delta_0)$, then for almost all $s\in[d,T]$ the energy level $\{H_g=s\}$ carries a periodic orbit of $X_{H_g,\delta\sigma}$ whose projection to $M$ represents $\alpha$.
\end{corollary}

Similarly, Theorem~\ref{thm:AET2} implies the following.
\begin{corollary}\label{cor:AET2}
Let $(M,g)$ be a closed Riemannian manifold of negative curvature, and let $\alpha\in[S^1,M]$ be any non-trivial free homotopy class. Denote by $\sigma$ any closed $2$-form on $M$. Then for $0<d<T$, there exists a positive constant $\delta_0=\delta_0(g,\alpha,d,T)$ such that if $|\sigma|_g<\delta_0$, then for almost every $s\in[d,T]$ the energy level $\{H_g=s\}$ carries a periodic orbit of $X_{H_g,\sigma}$ whose projection to $M$ represents $\alpha$.
\end{corollary}

\begin{remark}
{\rm
The assertions of the above corollaries overlap with Merry's results~\cite{Me0}. Indeed, Corollary~\ref{cor:AET1} complements~\cite[Theorem 1.1]{Me0} in the case that magnetic field $\sigma$ has no bounded primitive in the universal cover. On the other hand, \cite[Theorem 1.1]{Me0} implies Corollary~\ref{cor:AET2}. Merry actually showed a stronger result whenever $M$ admits a metric $g$ of negative curvature:
for each $d>0$ and each nontrivial homotopy class $\alpha\in[S^1,M]$, there is a constant $c=c(g,d)>0$ such that for every $s\in(d,\infty)$ and every closed $2$-form $\sigma$ with $|\sigma|_g<c$ the energy level $\{H_g=s\}$ carries a periodic orbit of $X_{H_g,\sigma}$ whose projection to $M$ represents $\alpha$.
}
\end{remark}

\begin{remark}
{\rm When $M$ is simply connected and $\sigma$ is a non-exact closed $2$-form on $M$,  our method in this paper does not apply to the twisted cotangent bundle $(T^*M,\omega_\sigma)$ since $\omega_\sigma$ is not weakly exact and hence the Floer action functional is not single-valued any more in this case. Fortunately,  very recently this weakly exact assumption is removed  in the context of a general symplectic cohomology for magnetic cotangent bundles, see~\cite{BR,GM}. In particular, the case $M=S^2$ is also considered in~\cite{BR} by Benedetti and Ritter.

}
\end{remark}

The paper is organized as follows. In Section~\ref{sec:2} we recall some background results and definitions of the filtered Floer homology on the twisted cotangent bundle. Section~\ref{sec:3} is mainly devoted to prove the invariance of Floer homology for symplectic deformations.
Before the proof,
the action spectrum properties (see Lemma~\ref{lem:NAS1} and Lemma~\ref{lem:NAS2}) are established in this section. The goal of Section~\ref{sec:4} is to compute the Floer homology in $T^*M$ with its canonical symplectic form $\omega_0$.
In Section~\ref{sec:5} we prove the main theorems and discuss the flows without non-contractible periodic orbits.

\section*{Acknowledgments}
I am deeply grateful to my Ph.D. advisor Guangcun Lu for introducing me to symplectic geometry and for valuable suggestions. My special thanks go to my colleague Xingpeng Dong for teaching me to draw beautiful graphs. I thank Rongrong Jin and Kun Shi for helpful discussions. I am also grateful to Viktor L. Ginzburg for useful comments and especially for his ideas in the proof of  Proposition~\ref{pro:nonexistence}. Finally, 
I am greatly indebted to the anonymous referee for the very carefully reading and helpful suggestions to improve the paper and for pointing out how to obtain Theorem~\ref{thm:3'} and Theorem~\ref{thm:4'}.

\section{Floer homology}\label{sec:2}
\setcounter{equation}{0}

\subsection{Preliminaries}

The Riemannian metric $g$ on $M$ induces a metric $\m$ on $TM$  and a horizontal-vertical splitting of $TT^*M$, together with isomorphisms
$$T_zT^*M=T^h_zT^*M\oplus T_z^vT^*M\cong T_qM\oplus T_q^*M\cong T_qM\oplus T_qM,\quad z=(q,p)\in T^*M.$$
The above splitting gives rise to the almost complex structure $J_g$ on $T^*M$ represented by
\begin{eqnarray}\label{e:Jg}
J_g=\begin{pmatrix}
 0 & -I \\ I & 0
\end{pmatrix}.
\end{eqnarray}
Recall that an almost complex structure $J$ on a symplectic manifold $(W,\omega)$ is  \emph{$\omega$-compatible} if the bilinear form $\omega(\cdot,J\cdot)$ defines a Riemannian metric on $W$.
It is easy to check that $J_g$ is $\omega_0$-compatible for every Riemmanian metric $g$ on $M$.  Denote $G_g(\cdot,\cdot):=\omega_0(\cdot,J_g\cdot)$. When $W$ is a manifold with a Riemanian metric $G$, various $L^\infty$-norms $|\cdot|_G$ are defined by
$$|v|_{G}:=\sup\limits_{x\in W}\|v(x)\|_G\quad \forall \;v\in \Gamma(TW),\quad |\theta|_g:=\sup_{x\in W}\|\theta(x)\|_g\quad \forall \;\theta\in\Omega^k(W),$$
$$|X|_{G}:=\sup\limits_{x\in W}\sup\{\|X(x)v\|_G\big|v\in T_xW,\;\|v\|_G=1\}\quad \forall \;X\in\Gamma({\rm End} (TW)).$$
Let $\mathscr{J}(T^*M)$ be the set of one-periodic almost complex structures on $T^*M$ with finite $|\cdot|_{G_g}$-norm, and denote $$\mathscr{J}(\omega_\sigma):=\{J\in\mathscr{J}(T^*M):J\;\hbox{is $\omega_\sigma$-compatible on $(T^*M,\omega_\sigma)$}\}.$$

Floer homology is the main tool utilized in this paper to prove the existence of periodic orbits of the Hamiltonian flow in twisted bundles. To define the Floer homology of compactly supported functions on a non-compact symplectically aspherical manifold we need to impose certain conditions on the manifold at infinity.
\begin{definition}
{\rm
We say that a symplectic manifold $(W,\omega)$ without boundary is geometrically bounded if there exists an almost complex structure $J$ and a complete Riemannian metric $g$ on $W$ such that
\begin{enumerate}
  \item $J$ is uniformly $\omega$-tame, which means that for all tangent vectors $X$ and $Y$ to $M$
      $$\omega(X,JX)\geq \kappa_1 \|X\|_g^2\quad\hbox{and}\quad|\omega(X,Y)|\leq \kappa_2\|X\|_g\|Y\|_g$$
      for some positive constants $\kappa_1$ and $\kappa_2$.
  \item the injectivity radius of $(W,g)$ is bounded away from zero and the sectional curvature of $(W,g)$ is bounded from above.
\end{enumerate}
}
\end{definition}
\noindent Obviously, closed symplectic manifolds are geometrically bounded; a product of two geometrically bounded sympletic manifolds is such a manifold. For a more detailed discussion of this concept please refer to Chapters V (by J.-C. Sikorav) and X (by M. Audin, F. Lalonde and L. Polterovich) in \cite{AL}. One can easily check that manifolds convex at infinity, e.g., $(\mathbb{R}^{2m},\sum_{i=1}^{2m}dx_i\wedge dy_i)$ and the cotangent bundle $(T^*M,\omega_0)$, are geometrically bounded. It is also well known that every twisted cotangent bundle $(T^*M,\omega_\sigma)$ admits abundant almost complex structures such that it is geometrically bounded (see \cite[Proposition~2.2]{CGK}). In general $J_g\notin \mathscr{J}(\omega_\sigma)$. However, Proposition~4.1 in \cite{Lu} implies that there exists a constant $\varepsilon_0=\varepsilon_0(g)>0$ such that for all $r\geq\varepsilon_0|\sigma|_{g}$,
$$\mathscr{J}(\omega_\sigma)\cap B_{J_g}(r)\neq\emptyset,$$
 where $B_{J_g}(r)$ denotes the open ball of radius $r$ about $J_g$ in $\mathscr{J}(T^*M)$. In fact, Lu in~\cite{Lu} shows that for $r\geq\varepsilon_0|\sigma|_{g}$ we can find an almost complex structure in $\mathscr{J}(\omega_\sigma)\cap B_{J_g}(r)$ such that $(T^*M,\omega_\sigma)$ for the natural metric $G_g$ is geometrically bounded; see also \cite{Me}. This will be very useful in the proof of Theorem~\ref{thm:Invariance} and Theorem~\ref{thm:Inv} since we will use it to choose suitable almost complex structures to estimate the energy of Floer cylinders and to ensure the necessary compactness of moduli spaces of solutions of Floer equations on twisted cotangent bundles.

Finally let us note that the first Chern class satisfies $c_1(T^*M,J)=0$ for each $J\in\mathscr{J}(\omega_\sigma)$ since the twisted cotangent bundle $(T^*M,\omega_\sigma)$ admits a Lagrangian distribution $T^vT^*M$ (see \cite{Se,Me}).

\subsection{The definition of filtered Floer homology}\label{subsec:FFH}
Let $\sigma \in \mathcal{P}(M)$, and let $\alpha\in [S^1,M]$ be a $\sigma$-atoroidal class. We denote $\mathscr{H}$ as the space of smooth compactly supported Hamiltonian functions on $S^1\times D_RT^*M$. For $c>0$ we denote by $\mathscr{H}_c$ the subspace of all Hamiltonian functions $H\in\mathscr{H}$ satisfying $\sup_{S^1\times M}H\leq -c$. Let $\mathfrak{L}_\alpha T^*M$ be the set of all $1$-periodic loops $x$ whose projections to $M$ belong to $\mathfrak{L}_\alpha M$. Fix a reference loop $q_\alpha\in\mathfrak{L}_\alpha M$.
We define the action functional $\mathscr{A}_{H,\sigma}:\mathfrak{L}_\alpha T^*M\to \mathbb{R}$ by
\begin{equation}\label{e:AF}
\mathscr{A}_{H,\sigma}(x)=\int_{S^1}x^*\lambda-\int_{[0,1]\times S^1} w^*\sigma-\int^1_0H(t,x)dt,
\end{equation}
where $w:[0,1]\times S^1\to M$ is any smooth map such that
$$w(0,t)=q_\alpha(t)\quad\hbox{and}\quad w(1,t)=\pi(x(t)).$$
\noindent Since $\alpha$ is a $\sigma$-atoroidal class,
$$\mathscr{A}_{\sigma}(q):=\int_{[0,1]\times S^1} w^*\sigma\quad
\forall\;q\in\mathfrak{L}_\alpha M$$
is independent of the choice of $w$, and therefore  $\mathscr{A}_{H,\sigma}$ is well defined. It is easy to check that the set ${\rm Crit}\mathscr{A}_{H,\sigma}$ of critical points of $\mathscr{A}_{H,\sigma}$ equals $\mathscr{P}_{\alpha}(H,\sigma)$. The set of values of $\mathscr{A}_{H,\sigma}$ on $\mathscr{P}_{\alpha}(H,\sigma)$ is called the \emph{action spectrum} with respect to $\alpha$, and we denote it by
$$\mathscr{S}_\alpha(H,\sigma)=\{\mathscr{A}_{H,\sigma}(x)\big|
x\in \mathscr{P}_{\alpha}(H,\sigma)\}.$$
Consider the space $\mathscr{H}$ with the strong Whitney $C^\infty$-topology. Note that the action spectrum $\mathscr{S}_\alpha(H,\sigma)$ is a compact and measure zero subset of $\mathbb{R}$ for any $H\in \mathscr{H}$, and is lower semicontinuous as a multivalued function of $H\in\mathscr{H}$ (see \cite[Section~4.4]{BPS}).
For $a,b\in \mathbb{R}\cup\{\pm\infty\}$, if $a,b\notin \mathscr{S}_\alpha(H,\sigma)$, we set
$$\mathscr{P}^{(a,b)}_{\alpha}(H,\sigma)=\{x\in \mathscr{P}_{\alpha}(H,\sigma)\big|a<\mathscr{A}_{H,\sigma}(x)<b\}.$$
\noindent In order to define the filtered Floer homology associated to $H,\sigma$ and $\alpha$ we need the following nondegeneracy condition:
\begin{description}
  \item {(C)}\quad Every element $x\in\mathscr{P}^{(a,b)}_{\alpha}(H,\sigma)$ is non-degenerate, that is, the linear map $d\phi_1^{H,\sigma}(x(0))$ does not have $1$ as an eigenvalue, where $\phi_1^{H,\sigma}$ is the time-one map of the flow of $X_{H,\sigma}$.
\end{description}
For any pair $a<b$ we consider the class of \emph{admissible Hamiltonians} by
$$\mathscr{H}_{\sigma;\alpha}^{a,b}=\{H\in\mathscr{H}\big|a,b\notin \mathscr{S}_\alpha(H,\sigma)\}.$$
If $\alpha=0$, then every point in the complement of $\bigcup_{t\in S^1}{\rm supp}(H_t)$ is a degenerate $1$-periodic orbit of $X_{H,\sigma}$ with zero action of $\mathscr{A}_{H,\sigma}$. To avoid these trivial periodic orbits, we require that $0\notin[a,b]$ whenever $\alpha=0$.

Given $H\in\mathscr{H}_{\sigma;\alpha}^{a,b}$ satisfying nondegeneracy condition (C), for every $x=(q(t),p(t))\in\mathscr{P}^{(a,b)}_{\alpha}(H,\sigma)$ we define the index $\mu(x)=-\mu_{CZ}(x)+\nu(x)$ following the paper by Weber \cite{We1}. Here $\mu_{CZ}(x)$ denotes the Conley-Zehnder index of $x$ (see \cite{Sa,SZ}), and $\nu(x):=0$ if the pullback bundle $q^*TM$ over $S^1$ is trivial and $\nu(x):=1$ otherwise. Consider the $\mathbb{Z}_2$-vector space ${\rm CF}^{(a,b)}_{\alpha}(H,\sigma)$\footnotemark[2]\footnotetext[2]{We use the convention that the complex generated by the empty set is zero.} defined by
$${\rm CF}^{(a,b)}_{\alpha}(H,\sigma):={\rm CF}^b_{\alpha}(H,\sigma)/{\rm CF}^a_{\alpha}(H,\sigma),\quad {\rm CF}^a_{\alpha}(H,\sigma):=\bigoplus\limits_{x\in\mathscr{P}^{(-\infty,a)}
_{\alpha}(H,\sigma)}\mathbb{Z}_2x$$
graded by the index $\mu$.

The Floer boundary operator  is defined as follows. Let $J_{gb}$ be an almost complex structure such that $(T^*M,\omega_\sigma)$ is geometrically bounded. Denote by $\mathcal{J}$ the set of smooth time-dependent $\omega_\sigma$-tame almost complex structures on $T^*M$ that are compatible with $\omega_\sigma$ near supp$(H)$ and equal to $J_{gb}$ outside some compact set. Every $J_t\in\mathcal{J}$ give rises to a positive-definite bilinear form on $\mathfrak{L}_\alpha T^*M$.
Given $x_{\pm}\in {\rm CF}^{(a,b)}_{\alpha}(H,\sigma)$ we denote by $\mathcal{M}^\alpha(x_-,x_+,H,J,\sigma)$ the \emph{moduli space} of smooth solution $u:\mathbb{R}\times S^1\to T^*M$ of the Floer differential equation
\begin{equation}\label{e:FDE}
\partial_su+J_t(u)(\partial_tu-X_{H,\sigma}(u))=0
\end{equation}
with the asymptotic boundary conditions
\begin{equation}\label{e:ABC}
\lim\limits_{s\to\pm\infty}u(s,t)=x_{\pm}\quad\hbox{and}\quad
\lim\limits_{s\to\pm\infty}\partial_su(s,t)=0
\end{equation}
uniformly in $t\in S^1$. For every solution of (\ref{e:FDE}) and
(\ref{e:ABC}) we have the energy identity
\begin{equation}\label{e:EI}
E(u):=\int^\infty_{-\infty}\int_0^1\omega_\sigma(\partial_su,
J\partial_su)dsdt=\mathscr{A}_{H,\sigma}(x_-)-\mathscr{A}_{H,\sigma}(x_+).
\end{equation}
Now we observe:
\begin{itemize}
  \item[(i)] The moduli space $\mathcal{M}^\alpha(x_-,x_+,H,J,\sigma)$ is uniformly $C^0$-bounded. This is because $H$ is compactly supported, and $(T^*M,\omega_\sigma)$ with $J_{gb}$ is geometrically bounded; see Chapter V in \cite{AL} or \cite{CGK,Lu}.
  \item[(ii)] Since $\omega_\sigma$ is symplectically aspherical, no bubbling off of holomorphic spheres can occur in $T^*M$. From this fact, the energy identity (\ref{e:EI}) and (i) we deduce that the moduli space  $\mathcal{M}^\alpha(x_-,x_+,H,J,\sigma)$ is compact with respect to $C^\infty$-convergence on compact sets.
  \item[(iii)] For a dense subset $\mathcal{J}_{reg}(H,\sigma)\subset\mathcal{J}$, the linearized operator for equation~(\ref{e:FDE}) is surjective for each finite-energy solution of (\ref{e:FDE}) in the homotopy class $\alpha$ (see \cite{FHS}).

\end{itemize}

For each $H$ satisfying (C), each $J\in\mathcal{J}_{reg}(H,\sigma)$ and each pair $x_{\pm}\in {\rm CF}^{(a,b)}_{\alpha}(H,\sigma)$ the moduli space $\mathcal{M}^\alpha(x_-,x_+,H,J,\sigma)$ is an empty set or a smooth manifold of dimension $\mu(x_+)-\mu(x_-)=\mu_{CZ}(x_+)-\mu_{CZ}(x_-)$; see \cite{Sa}. As usual, the Floer boundary operator $\partial=\partial^{H,J}_{\sigma;\alpha}$ on ${\rm CF}^{(a,b)}_{\alpha}(H,\sigma)$ is defined by

\begin{equation}\label{e:FBO}
\partial x_-=\sum\limits_{\substack{x_+\in {\rm CF}^{(a,b)}_{\alpha}(H,\sigma)
\\ \mu(x_-)-\mu(x_+)=1}}n(x_-,x_+)x_+.
\end{equation}
Here $n(x_-,x_+)$ stands for the number (mod $2$) of elements in the set $\mathcal{M}^\alpha(x_-,x_+,H,J,\sigma)/\mathbb{R}$ (modulo time shift). The operator $\partial$ satisfies $\partial\circ\partial=0$, and the resulting Floer homology groups
\begin{equation}\label{e:FHG}
{\rm HF}^{(a,b)}_{\alpha}(H,\sigma)=\frac{\rm ker\partial}{\rm im\partial}
\end{equation}
are independent of the choice of $J\in\mathcal{J}_{reg}(H,\sigma)$.

\begin{remark}
{\rm
It is unclear whether or not the Floer homology ${\rm HF}^{(a,b)}_{\alpha}(H,\sigma)$ is always independent of the choice of $J_{gb}$. It is independent of the choice of $J_{gb}$ whenever the set of almost complex structures for which $(T^*M,\omega_\sigma)$ is geometrically bounded is connected.
}
\end{remark}

 Assume that $a<b<c$, $0\notin [a,c]$ and $a,b,c\notin \mathscr{S}_\alpha(H,\sigma)$. Then we have the exact sequence of complexes
 \begin{equation}\label{e:esc}
0\to {\rm CF}^{(a,b)}_{\alpha}(H,\sigma)\to{\rm CF}^{(a,c)}_{\alpha}(H,\sigma)\to {\rm CF}^{(b,c)}_{\alpha}(H,\sigma)\to 0.
\end{equation}
This induces the exact sequence at the homology level
\begin{equation}\label{e:esFH}
\cdots\to {\rm HF}^{(a,b)}_{\alpha*}(H,\sigma)\to {\rm HF}^{(a,c)}_{\alpha*}(H,\sigma)\to {\rm HF}^{(b,c)}_{\alpha*}(H,\sigma)\to {\rm HF}^{(a,b)}_{\alpha(*-1)}(H,\sigma)\to\cdots.
\end{equation}

\subsection{Homotopic invariance}
Suppose that $H^{\pm}\in \mathscr{H}_{\sigma;\alpha}^{a,b}$ satisfy (C) and $x_{\pm}\in \mathscr{P}_{\alpha}(H^{\pm},\sigma)$. Let $H^s:\mathbb{R}\to \mathscr{H}$ be a smooth homotopy connecting $H^-$ and $H^+$ such that $H^s=H^-$ for $s\leq 0$ and $H^s=H^+$ for $s\geq 1$. Consider the parameter-dependent Floer equation
\begin{equation}\label{e:PFE}
\partial_su+J^s_t(u)(\partial_tu-X_{H^s_t,\sigma}(u))=0
\end{equation}
which satisfies uniformly in $t\in S^1$ the asymptotic boundary conditions
\begin{equation}\label{e:SABC}
\lim\limits_{s\to\pm\infty}u(s,t)=x_{\pm}\quad\hbox{and}\quad
\lim\limits_{s\to\pm\infty}\partial_su(s,t)=0.
\end{equation}
Here $J^s_t:\mathbb{R}\to \mathcal{J}$ is a \emph{regular homotopy} of smooth families of almost complex structures satisfying
\begin{itemize}
  \item $J^s_t=J^-_t\in\mathcal{J}_{reg}(H^-,\sigma)$ for $s\leq0$.
  \item $J^s_t=J^+_t\in\mathcal{J}_{reg}(H^+,\sigma)$ for $s\geq1$.
  \item $J^s_t$ is constant and and equals to $J_{gb}$ outside some compact set of $D_RT^*M$.
  \item The linearized operator for equation~(\ref{e:PFE}) is surjective for each finite-energy solution of (\ref{e:PFE}) in the homotopy class $\alpha$.
\end{itemize}
The moduli space $\mathcal{M}^\alpha(x_-,x_+,H^s,J^s,\sigma)$ of smooth solutions of (\ref{e:PFE}) satisfying the boundary conditions (\ref{e:SABC}) is $C_{loc}^\infty$-compact. A crucial ingredient for the proof of the compactness is the following energy identity:
\begin{eqnarray}\label{e:PEI}
E(u):&=&\int^\infty_{-\infty}\int_0^1\omega_\sigma(\partial_su,
J^s_t\partial_su)dsdt\nonumber\\
&=&\mathscr{A}_{H^-,\sigma}(x_-)-\mathscr{A}_{H^+,\sigma}(x_+)
-\int^\infty_{-\infty}\int^1_0(\partial_s H^s)(t,u(s,t))dsdt.
\end{eqnarray}
For $|H^+-H^-|$ small enough we can define a chain map (being similar to the argument in~\cite[Section~4.4]{BPS})
$$\widetilde{\Psi}^{\sigma}_{H^+H^-}:{\rm CF}^{(a,b)}_{\alpha}(H^-,\sigma)\to {\rm CF}^{(a,b)}_{\alpha}(H^+,\sigma)$$
which induces an isomorphism
$$\Psi^\sigma_{H^+H^-}:{\rm HF}^{(a,b)}_{\alpha}(H^-,\sigma)\to {\rm HF}^{(a,b)}_{\alpha}(H^+,\sigma).$$
The isomorphism $\Psi^\sigma_{H^+H^-}$ is independent of the choice of the homotopy $H^s$ and $J^s$ by a homotopy of homotopies argument; see \cite{Sa,SZ}. As a result, we can define the Floer homology groups
${\rm HF}^{(a,b)}_{\alpha}(H,\sigma)$ for any $H\in \mathscr{H}_{\sigma;\alpha}^{a,b}$ by a small perturbation since the Hamiltonians satisfying (C) for $a<b$ are dense in $\mathscr{H}_{\sigma;\alpha}^{a,b}$.

\subsection{Monotone homotopies}\label{subsec:MH}
Let $H,K\in \mathscr{H}_{\sigma;\alpha}^{a,b}$ be two functions with $H(t,x)\leq K(t,x)$ for all $(t,x)\in S^1\times D_RT^*M$. Choose a monotone homotopy $s\to H^s\in \mathscr{H}$ from $H$ to $K$ such that $\partial_s H^s\geq0$ everywhere (Here we do not require $H^s$ to be in $\mathscr{H}_{\sigma;\alpha}^{a,b}$ for every $s\in [0,1]$). From the energy identity (\ref{e:PEI}) we deduce that such a homotopy induces a natural homomorphism, which is called \emph{monotone homomorphism}
\begin{equation}\label{e:MH}
\Psi_{KH}^\sigma:{\rm HF}^{(a,b)}_{\alpha}(H,\sigma)\to {\rm HF}^{(a,b)}_{\alpha}(K,\sigma).
\end{equation}
It is well known that these monotone homomorphisms are independent of the choice of the monotone homotopy of Hamiltonians and satisfy the following properties
(see, e.g., \cite{BPS,FH,CFH,SZ,Vi}):
\begin{lemma}\label{lem:mh}
\begin{equation}
\begin{split}
\Psi^\sigma_{HH}={\rm id}\quad\forall\;H\in \mathscr{H}_{\sigma;\alpha}^{a,b},\\
\Psi^\sigma_{KH}\circ\Psi^\sigma_{HG}=\Psi^\sigma_{KG}
\end{split}
\end{equation}
whenever $ G,H,K\in \mathscr{H}_{\sigma;\alpha}^{a,b}$ satisfy $G\leq H\leq K$.
\end{lemma}
\begin{lemma}[{\rm see~\cite{Vi} or \cite[Section~4.5]{BPS}}]\label{lem:iso}
If $K^s$ is a monotone homotopy from $H$ to $K$ such that
$K^s\in \mathscr{H}_{\sigma;\alpha}^{a,b}$ for every $s\in[0,1]$, then $\Psi^\sigma_{KH}$ is an isomorphism.
\end{lemma}

\begin{remark}\label{rem:HFfcsf}
{\rm
To simplify the notation, we omit the mark $0$ whenever $\sigma=0$, and abbreviate, for example, $\mathscr{A}_{H,0}$, $\mathscr{H}_{0;\alpha}^{a,b}$,
${\rm HF}^{(a,b)}_{\alpha}(H,0)$ and $\Psi^0_{H^{+}H^{-}}$ by $\mathscr{A}_{H}$, $\mathscr{H}_{\alpha}^{a,b}$, ${\rm HF}^{(a,b)}_{\alpha}(H)$ and $\Psi_{H^{+}H^{-}}$ respectively.
All the above arguments in this section go through word for word whenever the closed two-form $\sigma$ vanishes identically . In this case, there is no restriction on the free homotopy class $\alpha\in[S^1,M]$, and the Floer homology for $H\in\mathscr{H}_{\alpha}^{a,b}$ defined here is exactly as that in Weber's paper~\cite{We0}. Indeed, $J_g\in\mathscr{J}(\omega_0)$ is an almost complex structure such that $(T^*M,\omega_0)$ is convex at infinity, and hence is geometrically bounded for $G_g$. Using  \cite[Proposition~2.3]{We0}, Weber defines the filtered Floer homology of a broader class of admissible Hamiltonians:

\begin{eqnarray}
\mathscr{K}_{R;\alpha}^{a,b}:=&\big\{&H\in C^\infty(S^1\times T^*M)\big|\exists\tau\geq0\;\exists c\in\mathbb{R}\;\hbox{such that}\;H_t(q,p)=\tau\|p\|_g+c\;\hbox{if}\;\notag\\
&&\|p\|_g\geq R\;\hbox{with}\;\{a,b\}\cap \mathscr{S}_\alpha(H)=\emptyset,\;\hbox{and}\;\tau\notin\Lambda_\alpha\;
\hbox{or}\;c\notin[a,b]\big\}\notag
\end{eqnarray}
which satisfies $\mathscr{K}_{R;\alpha}^{a,b}\supseteq\mathscr{H}_{\alpha}^{a,b}$.
Here we emphasize that Lemma~\ref{lem:iso} also holds for monotone homotopies in $\mathscr{K}_{R;\alpha}^{a,b}$; see \cite{We0}. This will be very useful to compute Floer homology in Section~\ref{sec:4}.
}
\end{remark}

\section{Symplectic deformations of Floer homology}\label{sec:3}
\setcounter{equation}{0}
Floer's work~\cite{Fl0,Fl1,Fl2} tells us that Floer homology is a topological invariant of a closed symplectically aspherical  manifold on which different symplectic forms give rise to the same Floer homology (up to isomorphisms). A direct proof of such a fact can be found in the paper by Viterbo~\cite{Vi}. In this section, by following the idea of Bae and Frauenfelder~\cite{BF}, we discuss the continuation homomorphisms for symplectic deformations under additional hypotheses concerning the sympletic forms. For related results about Floer homology under symplectic perturbations, we refer to the paper by Ritter~\cite{Ri}.

\subsection{Quadratic isoperimetric inequality}
\begin{lemma}\label{lem:QII}
Assume that $\sigma \in \mathcal{P}(M)$ and $\alpha\in[S^1,M]$ is a $\sigma$-atoroidal class. Then for every $q\in\mathfrak{L}_\alpha M$ we have
$$\big|\mathscr{A}_{\sigma}(q)\big|\leq \epsilon_0\bigg(
\int^1_0\|\dot{q}(t)\|_gdt\bigg)^2+\epsilon_1,$$
where $\epsilon_0=\epsilon_0(M,g,\sigma)>0$ and $\epsilon_1=\epsilon_1(M,g,\sigma,\alpha)>0$ are some constants.
\end{lemma}

This lemma is based on Lemma~2.7 in \cite{BF}; for a proof of it we refer to~\cite[Lemma~3]{FMP0}.

\begin{remark}\label{rem:lgc}
{\rm It is easy to check that $\mathcal{P}(M)$ is a linear space. Moreover, if $\sigma\in\mathcal{P}(M)$, for any $\delta\in\mathbb{R}$ it holds that
$$\big|\mathscr{A}_{\delta\sigma}(q)\big|\leq |\delta|\epsilon_0\bigg(
\int^1_0\|\dot{q}(t)\|_gdt\bigg)^2+|\delta|\epsilon_1,$$
where the constants $\epsilon_0$ and $\epsilon_1$ are given as in Lemma~\ref{lem:QII}. By Remark~\ref{rem:LGC}, for every closed Riemannian manifold $(M,g)$ of negative curvature, the constants $\epsilon_0(M,g,\sigma)$ and $\epsilon_1(M,g,\sigma,\alpha)$ in Lemma~\ref{lem:QII} converge to zero as $|\sigma|_g\to 0$.

}
\end{remark}

\subsection{Gap estimates for action spectrum}

\begin{lemma}\label{lem:NAS1}
Let $H\in\mathscr{H}$.
Suppose that $a\in\mathbb{R}$ is not in the action spectrum of $\mathscr{A}_{H}$. Then there exist some constants $\varepsilon_0=\varepsilon_0(H,g,\sigma,a,\alpha)>0$ and $\delta_0=\delta_0(H,g,\sigma,a,\alpha)>0$ such that if $|\delta|<\delta_0$, then
$[a-\varepsilon_0,a+\varepsilon_0]\cap\mathscr{S}_\alpha(H,\delta\sigma)
=\emptyset$.
\end{lemma}
\begin{proof}
 Arguing by contradiction, suppose that there is a sequence of number $\{\delta_k\}_{k\in \mathbb{N}}\subseteq \mathbb{R}$ such that $|\delta_k|<1/(k|\sigma|_g)$ (if $\sigma=0$ let $\delta_k=0$ for all $k$), and for every $\sigma_k:=\delta_k\sigma$ there exists some $x_k\in\mathscr{P}_{\alpha}(H,\sigma_k)$, that is,
 \begin{equation}\label{e:hpo}
\dot{x}_k=X_{H,\sigma_k}(t,x_k)\quad \hbox{with}\quad x_k(0)=x_k(1),
\end{equation}
such that $\mathscr{A}_{H,\sigma_k}(x_k)=a_k\in(a-1/k,a+1/k)$.
Since $$\iota(X_{H})\omega_0=dH_t=\iota(X_{H,\sigma_k})\omega_{\sigma_k}
=\iota(X_{H,\sigma_k})\omega_0+\iota(X_{H,\sigma_k})\pi^*\sigma_k,$$
we deduce that
\begin{equation}\notag
\omega_0\big(X_{H}-X_{H,\sigma_k},J_g(X_{H}-X_{H,\sigma_k})\big)
=\pi^*\sigma_k\big(X_{H,\sigma_k},J_g(X_{H}-X_{H,\sigma_k})\big),
\end{equation}
where $J_g$ is defined as in (\ref{e:Jg}). Then we have
\begin{equation}\notag
\|X_{H}-X_{H,\sigma_k}\|^2_{G_g}\leq\|\sigma_k\|_g
\|X_{H,\sigma_k}\|_{G_g}\|X_{H}-X_{H,\sigma_k}\|_{G_g}
\end{equation}
which implies that
\begin{equation}\label{e:ub}
\|X_{H}-X_{H,\sigma_k}\|_{G_g}\leq\|\sigma_k\|_g\|X_{H,\sigma_k}\|_{G_g}\quad
\hbox{and}\quad
\|X_{H,\sigma_k}\|_{G_g}\leq\frac{\|X_{H}\|_{G_g}}{1-\|\sigma_k\|_g}
\end{equation}
for $|\sigma_k|_g<1$. Combining this with the fact that $H$ is compactly supported in $S^1\times D_RT^*M$ we deduce that $X_{H,\sigma_k}$ is uniformly bounded, and hence $x_k(t)$ is equicontinuous. Then, by Arzela-Ascoli theorem, passing to a subsequence, $x_k(t)$ converges to some $x_0(t)$ in $D_RT^*M$. We claim that $x_0(t)$ is a Hamiltonian periodic orbit of $H$ for $\omega_0$. For $k$ large enough, we may assume without loss of generality that $H$ is defined on $\mathbb{R}^{2n}$ (by using some local  coordinate near $x_0(t)$). Now we only need to prove
\begin{equation}\notag
x_0(t)-x_0(0)=\int^t_0X_H(s,x_0(s))ds.
\end{equation}
Note that $$x_0(t)-x_0(0)=\lim_{k\to\infty}(x_k(t)-x_k(0))
=\lim_{k\to\infty}\int^t_0\dot{x}_k(s)ds.$$ We compute
\begin{eqnarray}\notag
x_0(t)-x_0(0)-\int^t_0X_H(s,x_0(s))ds&=&\lim\limits_{k\to\infty}
\int^t_0\big(\dot{x}_k(s)-X_H(s,x_0(s))\big)ds\nonumber\\
&=&\lim\limits_{k\to\infty}
\int^t_0\big(\dot{x}_k(s)-X_{H,\sigma_k}(s,x_k(s))\big)ds\nonumber\\
&&+\lim\limits_{k\to\infty}
\int^t_0\big(X_{H,\sigma_k}(s,x_k(s))-X_{H}(s,x_k(s))\big)ds\nonumber\\
&&+\lim\limits_{k\to\infty}
\int^t_0\big(X_{H}(s,x_k(s))-X_{H}(s,x_0(s))\big)ds.\nonumber
\end{eqnarray}
In the last equality, the first term is zero due to (\ref{e:hpo}), the second term is zero since $X_{H,\sigma_k}$ uniformly converges to $X_H$ by (\ref{e:ub}) and the compactness of supp$(H)$, and the third term is zero since $x_k(t)$  uniformly tends to $x_0(t)$. Let $q_k(t)=\pi(x_k(t))$. By lemma~\ref{lem:QII}, we have
\begin{eqnarray}\notag
\big|\mathscr{A}_{\sigma_k}(q_k)\big|&\leq& \delta_k\epsilon_0\bigg(
\int^1_0\|\dot{q}_k(t)\|_gdt\bigg)^2+\delta_k\epsilon_1\nonumber\\
&\leq&\delta_k\epsilon_0\bigg(
\int^1_0\|X_{H,\sigma_k}\|_{G_g}dt\bigg)^2+\delta_k\epsilon_1\to0\quad
\hbox{as} \;k\to0.\nonumber\\
\end{eqnarray}
It follows that
$$a=\lim\limits_{k\to\infty}\mathscr{A}_{H,\sigma_k}(x_k)=\lim\limits_{k\to\infty}\mathscr{A}_{H}(x_k)
-\lim\limits_{k\to\infty}\mathscr{A}_{\sigma_k}(q_k)=\mathscr{A}_{H}(x_0)$$
which contradicts our assumption that $a$ is not in the action spectrum of $\mathscr{A}_{H}$.
\end{proof}

By Remark~\ref{rem:lgc}, the proof of Lemma~\ref{lem:NAS1} implies

\begin{lemma}\label{lem:NAS2}
Let $(M,g)$ be a closed Riemannian manifold of non-positive curvature, and let $H\in\mathscr{H}$.
Suppose that $a\in\mathbb{R}$ is not in the action spectrum of $\mathscr{A}_{H}$, and that $\sigma$ is any closed $2$-form on $M$. Then there exists some constants $\delta_0=\delta_0(H,g,a,\alpha)>0$ and $\varepsilon_0=\varepsilon_0(H,g,a,\alpha)>0$ such that if $|\sigma|_g<\delta_0$, then $[a-\varepsilon_0,a+\varepsilon_0]\cap \mathscr{S}_\alpha(H,\sigma)=\emptyset$.
\end{lemma}

\begin{remark}\label{rem:ppo}
{\rm Under the hypotheses of  Lemma~\ref{lem:NAS1}, if, moreover,  $\{H_k\}_{k\in\mathbb{N}}\subseteq\mathscr{H}$ converges to $H$ in the $C^\infty$-topology, then we conclude that there exists a positive integer $k_0>0$ such that $[a-\varepsilon_0,a+\varepsilon_0]\cap\mathscr{S}_\alpha(H_k,\delta\sigma)
=\emptyset$ for every $\delta$ satisfying $|\delta|<\delta_0(H,g,\sigma,a,\alpha)$ and every $k\geq k_0$.
Similarly, under the hypotheses of  Lemma~\ref{lem:NAS2}, if $\mathscr{H}\ni H_k \stackrel{C^\infty}{\longrightarrow}H$, then there exists a positive integer $k_0>0$ such that if $k\geq k_0$ and $|\sigma|_g<\delta_0(H,g,a,\alpha)$, then $[a-\varepsilon_0,a+\varepsilon_0]\cap \mathscr{S}_\alpha(H_k,\sigma)=\emptyset$.

}
\end{remark}

\subsection{Invariance of Floer homology for symplectic deformations}

\begin{theorem}\label{thm:Invariance}
Assume that $H\in \mathscr{H}_{\alpha}^{a,b}$. Then the following holds.
\begin{description}
  \item[(1)] There exists a constant $\delta_0=\delta_0(H,g,\sigma,a,b,\alpha)>0$ such that if $|\delta|<\delta_0$, then there is a continuation chain map
$$\widetilde{\Psi_{\omega_0}^{\omega_{\delta\sigma}}}:{\rm CF}^{(a,b)}_{\alpha}(H)\to {\rm CF}^{(a,b)}_{\alpha}(H,\delta
\sigma)$$
which induces an isomorphism
\begin{equation}
\Psi_{\omega_0}^{\omega_{\delta\sigma}}:{\rm HF}^{(a,b)}_{\alpha}(H)\to {\rm HF}^{(a,b)}_{\alpha}(H,\delta
\sigma).
\end{equation}
  \item[(2)] If $K\in \mathscr{H}_{\alpha}^{a,b}$ is another Hamiltonian function satisfying $$H(t,x)\leq K(t,x)\quad \forall\;(t,x)\in [0,1]\times D_RT^*M,$$

then for $|\delta|<\min\{\delta_0(H,g,\sigma,a,b,\alpha),
\delta_0(K,g,\sigma,a,b,\alpha)\}$ the following diagram commutes:

\begin{eqnarray}
\begin{CD}\label{diag:dc0}
{\rm HF}^{(a,b)}_{\alpha}(H) @>{\Psi_{KH}}>> {\rm HF}^{(a,b)}_{\alpha}(K)\\
@V{\Psi_{\omega_0}^{\omega_{\delta\sigma}}}VV  @VV{\Psi_{\omega_0}^{\omega_{\delta\sigma}}}V \\
 {\rm HF}^{(a,b)}_{\alpha}(H,\delta\sigma) @>{\Psi^{\delta\sigma}_{KH}}>> {\rm HF}^{(a,b)}_{\alpha}(K,\delta\sigma)
\end{CD}
\end{eqnarray}
\end{description}

\end{theorem}

\begin{proof}
In what follows, we always assume that $|\sigma|_g\neq0$ (nothing needs to be proved if $\sigma=0$).
Consider $H^i\in\mathscr{H}_{\alpha}^{a,b}$, $i=0,1$, which satisfies
$$H^0(t,x)\leq H^1(t,x)\quad \forall\;(t,x)\in [0,1]\times D_RT^*M.$$
By Lemma~\ref{lem:NAS1}, there exist some constants $\varepsilon^i=\varepsilon^i(H^i,g,\sigma,a,b,\alpha)>0$ and $\hat{\delta}^i=\hat{\delta}(H^i,g,\sigma,a,b,\alpha)>0$, $
i=0,1$, such that for every $\delta^i\in(-\hat{\delta}^i,\hat{\delta}^i)$ it holds that
\begin{equation}
\begin{split}
\mathscr{S}_\alpha(H^i,\delta^i\sigma)\cap[a-2\varepsilon^i,a+2\varepsilon^i]=\emptyset
\quad \hbox{and}\quad \mathscr{S}_\alpha(H^i,\delta^i\sigma)\cap[b-2\varepsilon^i,b+2\varepsilon^i]=\emptyset.
\end{split}
\end{equation}
By a perturbation argument, we may assume that without loss of generality $H^0$ and $H^1$ satisfy the nondegeneracy condition (C) for $\omega_{\delta^0\sigma}$ and $\omega_{\delta^1\sigma}$ respectively. We will show that there is a Floer chain map from
${\rm CF}^{(a,b)}_{\alpha}(H^0,\delta^0\sigma)$ to ${\rm CF}^{(a,b)}_{\alpha} (H^1,\delta^1\sigma)$ which induces a homomorphism
$$\Psi^{\delta^1\delta^0}_{H^1H^0}:{\rm HF}^{(a,b)}_{\alpha}(H^0,\delta^0\sigma)\to {\rm HF}^{(a,b)}_{\alpha}(H^1,\delta^1\sigma)$$
whenever $|\delta^i|$ $(i=0,1)$ is small enough.

Let $\beta:\mathbb{R}\to [0,1]$ be a smooth cut-off function such that
$\beta=0$ for $s\leq0$, $\beta(s)=1$ for $s\geq 1$ and $0\leq\beta'(s)\leq1$.
Set $\delta^s=(1-\beta(s))\delta^0+\beta(s)\delta^1$ and
$H^s=(1-\beta(s))H^0+\beta(s)H^1$. Let $\omega^s=\omega_0+\delta^s\sigma$, and let $X^{\omega^s}_{H_t^s}$ be the Hamiltonian vector field such that $$dH_t^s=\iota_{X^{\omega^s}_{H_t^s}}\omega^s.$$
Then by~(\ref{e:ub}) we have
\begin{equation}\label{e:ubHv}
|X^{\omega^s}_{H_t^s}|_{G_g}\leq\frac{|X^{\omega_0}_{H_t^s}|_{G_g}}
{1-|\delta^s||\sigma|_g}
\leq2\min\big\{|X^{\omega_0}_{H_t^0}|_{G_g},|X^{\omega_0}_{H_t^1}|_{G_g}
\big\}
\end{equation}
for $|\delta^0|,|\delta^1|\leq1/(2|\sigma|_g)$ ($\sigma\neq 0$).
Let $s\to J^s\in\mathscr{J}(\omega_{\delta^s\sigma})\cap
B_{J_g}(\varepsilon_0|\delta^s\sigma|_{g})$ (recall that $\varepsilon_0$ is a constant given on page 10)
be a homotopy of one-periodic almost complex structures such that $J_t^s=J_t^-\in\mathcal{J}_{reg}(H^0,\delta^0\sigma)$ for $s\leq0$, and $J_t^s=J_t^+\in\mathcal{J}_{reg}(H^1,\delta^1\sigma)$ for $s\geq1$. Such a choice of $J^s$ is possible because $\delta^s$ is constant outside of $[0,1]$ and Proposition 4.1 in~\cite{Lu} can actually be extended to a parametric version for a family of twisted symplectic forms $\omega_{\sigma^s}$ with $s$ belonging to some compact interval since those almost complex structures constructed there are canonical with respect to a choice of metric. For every $(s,t)\in \mathbb{R}\times S^1$ and every $X,Y\in TT^*M$, we have
\begin{eqnarray}\label{e:utJs0}
\omega^s(X,J^s_tX)&=&\omega_0(X,J_gX)+\omega_0(X,(J^s_t-J_g)X)
+\delta^s\pi^*\sigma(X,J_t^sX)\nonumber\\
&\geq&\|X\|^2_{G_g}-|J^s_t-J_g|_{G_g}\|X\|^2_{G_g}-
|\delta^s||\sigma|_g|J^s_t|_{G_g}\|X\|^2_{G_g}\notag\\
&\geq&\big(1-\varepsilon_0|\delta^s||\sigma|_g-(1+
\varepsilon_0|\delta^s||\sigma|_g)|\delta^s||\sigma|_g\big)
\|X\|^2_{G_g}\notag\\
&\geq&\frac{1}{2}\|X\|^2_{G_g}
\end{eqnarray}
provided $|\delta^0|,|\delta^1|\leq (6\varepsilon_0+4)^{-1}|\sigma|_g^{-1}$, and it holds that
\begin{eqnarray}\label{e:utJs1}
|\omega^s(X,Y)|\leq\frac{5}{3}\|X\|_{G_g}\|Y\|_{G_g}.
\end{eqnarray}
Therefore $J^s$ is a $1$-periodic almost complex structure for which  $(T^*M,\omega^s)$ for the natural metric $G_g$ is geometrically bounded for every $s\in \mathbb{R}$ and every $t\in S^1$.

Given $x\in \mathscr{P}_{\alpha}(H^0,\delta^0\sigma)$ and $y\in \mathscr{P}_{\alpha}(H^1,\delta^1\sigma)$,
consider $u:\mathbb{R}\times S^1\to T^*M$ satisfying the following equation
\begin{equation}\label{e:sfe}
\partial_su+J^s_t(u)(\partial_tu-X^{\omega^s}_{H_t^s}(u))=0
\end{equation}
with the asymptotic boundary conditions
\begin{equation}\label{e:abc}
\lim\limits_{s\to-\infty}u(s,t)=x,\quad
\lim\limits_{s\to+\infty}u(s,t)=y\quad\hbox{and}\quad
\lim\limits_{s\to\pm\infty}\partial_su(s,t)=0.
\end{equation}
Here we emphasize that $\{J^s_t\}$ is also chosen such that solutions of (\ref{e:sfe}) and (\ref{e:abc}) are transverse (the associated linearized operators are surjective) by a perturbation argument; see~\cite{FHS}.
Now we can define a map
$$\widetilde{\Psi^{\delta^1\delta^0}_{H^1H^0}}:{\rm CF}^{(a,b)}_{\alpha}(H^0,\delta^0
\sigma)\to {\rm CF}^{(a,b)}_{\alpha}(H^1,\delta^1
\sigma)$$
given by
$$\widetilde{\Psi^{\delta^1\delta^0}_{H^1H^0}}(x)
=\sum\limits_{\substack{\mu(x)=\mu(y)}}\#_2
\mathcal{M}^\alpha(x,y,H^s,J^s,\omega^s)y,$$
where the space $\mathcal{M}^\alpha(x,y,H^s,J^s,\omega^s)$ consists of solutions of (\ref{e:sfe}) and (\ref{e:abc}), and $\#_2$ denotes the number of elements modulo two. Since $\omega_s$ is symplectically aspherical for every $s\in \mathbb{R}$, there is no bubbling.  Then in order to obtain the compactness of $\mathcal{M}^\alpha(x,y,H^s,J^s,\omega^s)$, we need certain energy estimations. Write
$$\mathscr{A}_{H,\delta(s)\sigma}(u(s,\cdot))=\int_{S^1}u(s,\cdot)^*
\lambda-\mathscr{A}_{\delta(s)\sigma}(q(s,\cdot))-\int^1_0H(t,u(s,t))dt,$$
where $q(s,t)$ is the projection of $u(s,t)\in T^*M$ to $M$ for every $(s,t)\in \mathbb{R}\times S^1$. Then we compute
\begin{eqnarray}\label{e:EIE}
\mathscr{A}_{H^1,\delta^1\sigma}(y)-\mathscr{A}_{H^0,\delta^0\sigma}(x)&=&
\int^\infty_{-\infty}\frac{d}{ds}\mathscr{A}_{H,\delta(s)\sigma}
(u(s,\cdot)){ds}\nonumber\\
&=&-\int^\infty_{-\infty}\int_0^1\omega_s(\partial_s u,J_t^s\partial_s u)ds dt-\int^\infty_{-\infty}\int_0^1\partial_sH^s_t(u(s,t))dsdt.\nonumber\\
&&-\int^\infty_{-\infty}(\delta^1-\delta^0)\beta^\prime(s)\mathscr{A}
_{\sigma}(q(s,\cdot))ds.
\end{eqnarray}
By Lemma~\ref{lem:QII} we estimate
\begin{eqnarray}\label{e:EIE0}
\big|\mathscr{A}
_{\sigma}(q(s,\cdot))\big|&\leq&\epsilon_0\bigg(\int^1_0
\|\partial_tq(s,t)\|_gdt\bigg)^2+\epsilon_1\nonumber\\
&\leq&\epsilon_0\int^1_0\|\partial_tq(s,t)\|^2_gdt+\epsilon_1\nonumber\\
&\leq&\epsilon_0\int^1_0\|\partial_tu(s,t)\|^2_{G_g}dt+\epsilon_1
\end{eqnarray}
for some positive constants $\epsilon_0=\epsilon_0(g,\sigma)$ and $\epsilon_1=\epsilon_1(g,\sigma,\alpha)$.  By (\ref{e:ubHv}) and (\ref{e:sfe}) we have
\begin{eqnarray}\label{e:EIE1}
\|\partial_tu\|^2_{G_g}&=&\|J^s_t(u)\partial_su+X_{H_t^s}^{\omega^s}(u)
\|^2_{G_g}\nonumber\\
&\leq&2\big(\|J^s_t(u)\partial_su\|^2_{G_g}+\|X_{H_t^s}^{\omega_s}(u)
\|^2_{G_g}\big)\nonumber\\
&\leq&2\big(|J^s_t|_{G_g}\|\partial_su\|^2_{G_g}+C_0\big)\nonumber\\
&\leq&2\bigg(\frac{4}{3}
\|\partial_su\|^2_{G_g}+C_0\bigg)
\end{eqnarray}
for $|\delta^0|,|\delta^1|\leq (6\varepsilon_0+4)^{-1}|\sigma|_g^{-1}$.
Here $C_0:=C_0(H^0,H^1,g)=4\min\{|X^{\omega_0}_{H^0}
|^2_{G_g},|X^{\omega_0}_{H^1}|^2_{G_g}\}$.
Plugging (\ref{e:utJs0}) into (\ref{e:EIE1}), we arrive at
\begin{eqnarray}\label{e:EIE2}
\|\partial_tu\|^2_{G_g}
&\leq&\frac{16}{3}\omega^s
(\partial_su,J^s_t\partial_su)+2C_0.
\end{eqnarray}
Combining (\ref{e:EIE0}) and (\ref{e:EIE2}) leads to
\begin{eqnarray}\label{e:EIE3}
\big|\mathscr{A}
_{\sigma}(q(s,\cdot))\big|\leq\frac{16\epsilon_0}{3}
\int^1_0\omega_s(\partial_su,J^s_t\partial_su)dt+2\epsilon_0C_0+\epsilon_1.
\end{eqnarray}
Hence,
\begin{eqnarray}\label{e:EIE4}
\bigg|\int^\infty_{-\infty}\beta^\prime(s)\mathscr{A}
_{\sigma}(q(s,\cdot))ds\bigg|&=&\bigg|\int^1_{0}\beta^\prime(s)\mathscr{A}
_{\sigma}(q(s,\cdot))ds\bigg|\notag\\
&\leq&C_1\int^\infty_{-\infty}\int^1_0\omega_s(\partial_su,J^s_t\partial_su)dsdt
+C_2,
\end{eqnarray}
where $C_1=16\epsilon_0/3$ and
$C_2=2\epsilon_0C_0+\epsilon_1$. Then by (\ref{e:EIE}) we obtain
\begin{eqnarray}\label{e:EIE5}
(1-\delta C_1)\int^\infty_{-\infty}\int^1_0\omega_s
(\partial_su,J^s_t\partial_su)dsdt&\leq&
\mathscr{A}_{H^0,\delta^0\sigma}(x)-\mathscr{A}_{H^1,\delta^1\sigma}(y)
+\delta C_2\nonumber\\
&&-\int^\infty_{-\infty}\int_0^1\partial_sH^s_t(u(s,t))dsdt,
\end{eqnarray}
with $\delta=(|\delta^0|+|\delta^1|)$.

Denote $\varepsilon:=\min\{\varepsilon^0,\varepsilon^1\}$, and set
$$\delta_0:=\delta_0(H^0,H^1,g,\sigma,a,b,\alpha):=\min\bigg\{\hat{\delta}^0, \hat{\delta}^1,\frac{1}{(6\varepsilon_0+4)|\sigma|_g},
\frac{1}{2C_1},\frac{\varepsilon}{2C_2}\bigg\}.$$
For brevity, hereafter we drop the explicit dependence on $a,b$ in the notation $\delta_0$. 
Then for $|\delta^i|<\delta_0$ the Floer map from ${\rm CF}_{\alpha}(H^0,\delta^0\sigma)$ to ${\rm CF}_{\alpha}(H^1,\delta^1\sigma)$ defined by the solutions of (\ref{e:sfe}) preserves the subcomplexes ${\rm CF}^a_{\alpha}$ and ${\rm CF}^b_{\alpha}$ (since $\partial_sH^s=\dot{\beta}(s)(H^1-H^0)\geq0$).
Therefore, the solutions of (\ref{e:sfe}) give rise to the continuation map $\widetilde{\Psi^{\delta^1\delta^0}_{H^1H^0}}$. By a standard gluing argument in Floer homology theory, $\widetilde{\Psi^{\delta^1\delta^0}_{H^1H^0}}$ commutes with the boundary operators. Hence $\widetilde{\Psi^{\delta^1\delta^0}_{H^1H^0}}$ is a chain map which induces a homomorphism
$$\Psi^{\delta^1\delta^0}_{H^1H^0}:{\rm HF}^{(a,b)}_{\alpha}(H^0,\delta^0\sigma)\to {\rm HF}^{(a,b)}_{\alpha}(H^1,\delta^1\sigma).$$
Suppose that
$$\Psi^{\delta^2\delta^1}_{H^2H^1}:{\rm HF}^{(a,b)}_{\alpha}(H^1,\delta^1\sigma)\to {\rm HF}^{(a,b)}_{\alpha}(H^2,\delta^2\sigma)$$
is another homomorphism defined as above if $H^1\leq H^2$ and
$|\delta^1|,|\delta^2|<\delta_0(H^1,H^2,g,$ $\sigma,\alpha)$. By a homotopy-of-homotopies argument, we have the following commutative diagram
\begin{equation}\label{e:ch'}
\xymatrix{{\rm HF}^{(a,b)}_{\alpha}(H^0,\delta^0\sigma)
\ar[rr]^{\Psi^{\delta^2\delta^0}_{H^2H^0}}\ar[dr]_{\Psi^{\delta^1\delta^0}_{H^1H^0}}& & {\rm HF}^{(a,b)}_{\alpha}(H^2,\delta^2\sigma)\\ & {\rm HF}^{(a,b)}_{\alpha}(H^1,\delta^1\sigma)\ar[ur]_{\Psi^{\delta^2\delta^1}_{H^2H^1}} & }
\end{equation}
for $|\delta^0|,|\delta^1|,|\delta^2|<\min\{\delta_0(H^0,H^1,g,\sigma,\alpha)
,\delta_0(H^1,H^2,g,\sigma,\alpha),\delta_0(H^0,H^2,$ $g,\sigma,\alpha)\}$.

If $H^s\equiv H\in\mathscr{H}_{\alpha}^{a,b}$, then for $|\delta^1|,|\delta^2|<\delta_0(H,H,g,\sigma,\alpha)$
 we can define the map
$$\Psi^{\delta^0\delta^1}_{HH}:{\rm HF}^{(a,b)}_{\alpha}(H,\delta^1\sigma)\to {\rm HF}^{(a,b)}_{\alpha}(H,\delta^0\sigma).$$
Since $\Psi^{\delta\delta}_{HH}$ is an isomorphism for every $\delta\in\mathbb{R}$ and every
$H\in\mathscr{H}_{\delta\sigma;\alpha}^{a,b}$, we deduce from (\ref{e:ch'}) with $\delta_0=\delta_2$ that the homomorphism
\begin{equation}\label{e:iso}
\Psi^{\delta^1\delta^0}_{HH}:{\rm HF}^{(a,b)}_{\alpha}(H,\delta^0\sigma)\to {\rm HF}^{(a,b)}_{\alpha}(H,\delta^1\sigma)
\end{equation}
is an isomorphism with inverse $\Psi^{\delta^0\delta^1}_{HH}$ for $|\delta^0|,|\delta^1|<\delta_0(H,H,g,\sigma,\alpha)$.

Now we are in a position to prove Theorem~\ref{thm:Invariance}. Denote $$\delta_0(H,g,\sigma,\alpha):=\delta_0(H,H,g,\sigma,\alpha).$$
Let $\delta^0=0$ and $\delta^1=\delta$ with $|\delta|<\delta_0(H,g,\sigma,\alpha)$ for every $H\in \mathscr{H}_{\alpha}^{a,b}$, and denote $$\Psi_{\omega_0}^{\omega_{\delta\sigma}}:=\Psi^{\delta^1\delta^0}_{HH}
:{\rm HF}^{(a,b)}_{\alpha}(H)\to {\rm HF}^{(a,b)}_{\alpha}(H,\delta\sigma).$$
Then statement~(1) follows from the isomorphism (\ref{e:iso}) immediately. Notice that
$$\min\{\delta_0(H,g,\sigma,\alpha),\;\delta_0(K,g,\sigma,\alpha)\}
\leq\delta_0(H,K,g,\sigma,\alpha).$$
For $|\delta|<\min\{\delta_0(H,g,\sigma,\alpha),
\delta_0(K,g,\sigma,\alpha)\}$ we deduce from (\ref{e:ch'}) that the homomorphism
$$\Psi^{\delta 0}_{KH}:{\rm HF}^{(a,b)}_{\alpha}(H)\to {\rm HF}^{(a,b)}_{\alpha}(K,\delta\sigma)$$
satisfies $\Psi^{\delta 0}_{KH}=\Psi^{\delta\sigma}_{KH}
\circ\Psi_{\omega_0}^{\omega_{\delta\sigma}}$ with $(H^0,\delta^0)=(H,0), (H^1,\delta^1)=(H,\delta)$ and $(H^2,\delta^2)=(K,\delta)$, and 
$\Psi^{\delta 0}_{KH}=\Psi_{\omega_0}
^{\omega_{\delta\sigma}}\circ\Psi_{KH}$ with $(H^0,\delta^0)=(H,0), (H^1,\delta^1)=(K,0)$ and $(H^2,\delta^2)=(K,\delta)$. So $\Psi^{\delta\sigma}_{KH}
\circ\Psi_{\omega_0}^{\omega_{\delta\sigma}}=\Psi_{\omega_0}
^{\omega_{\delta\sigma}}\circ\Psi_{KH}$. The proof of statement~(2) is completed.
\end{proof}

\begin{remark}\label{rem:inv}
{\rm
The quadratic isoperimetric inequality in Lemma~\ref{lem:QII} plays an essential role in the proof of the above theorem. A careful inspection of the proof of Theorem~\ref{thm:Invariance} shows that one could obtain the desired various estimations whenever the term $\mathscr{A}_{\sigma}$ associated to $\sigma \in \mathcal{P}(M)$ is well controlled. In fact, when the underlying closed manifold $M$ admits a metric $g$ of negative curvature, Theorem~\ref{thm:Invariance} can be upgraded to the following theorem.

}
\end{remark}

\begin{theorem}\label{thm:Inv}
Let $(M,g)$ be a closed Riemannian manifold of negative curvature, and let $\alpha\in[S^1,M]$ be a free homotopy class.
Assume that $H\in \mathscr{H}_{\alpha}^{a,b}$ and that $\sigma$ is any closed $2$-form on $M$. Then there exists a constant $\delta_0=\delta_0(H,g,a,b,\alpha)>0$ such that if $|\sigma|_g<\delta_0$, then there is a continuation chain map
$$\widetilde{\Psi_{\omega_0}^{\omega_{\sigma}}}:{\rm CF}^{(a,b)}_{\alpha}(H)\to {\rm CF}^{(a,b)}_{\alpha}(H,\sigma)$$
which induces an isomorphism
\begin{equation}
\Psi_{\omega_0}^{\omega_{\sigma}}:{\rm HF}^{(a,b)}_{\alpha}(H)\to {\rm HF}^{(a,b)}_{\alpha}(H,\sigma).
\end{equation}
Moreover, if $K\in \mathscr{H}_{\alpha}^{a,b}$ is another Hamiltonian function satisfying $$H(t,x)\leq K(t,x)\quad \forall\;(t,x)\in [0,1]\times D_RT^*M,$$  then for $|\sigma|_g<\min\{\delta_0(H,g,a,b,\alpha),
\delta_0(K,g,a,b,\alpha)\}$ the following diagram commutes:

\begin{eqnarray}
\begin{CD}\label{diag:dc1}
{\rm HF}^{(a,b)}_{\alpha}(H) @>{\Psi_{KH}}>> {\rm HF}^{(a,b)}_{\alpha}(K)\\
@V{\Psi_{\omega_0}^{\omega_{\sigma}}}VV  @VV{\Psi_{\omega_0}^{\omega_{\sigma}}}V \\
 {\rm HF}^{(a,b)}_{\alpha}(H,\sigma) @>{\Psi^\sigma_{KH}}>> {\rm HF}^{(a,b)}_{\alpha}(K,\sigma)
\end{CD}
\end{eqnarray}

\end{theorem}

\begin{proof}
Firstly, Remark~\ref{rem:NC} implies that under the hypotheses of Theorem~\ref{thm:Inv} the filtered Floer homology ${\rm HF}^{(a,b)}_{\alpha}(H,\sigma)$ can be defined as in Subsection~\ref{subsec:FFH}. As Remark~\ref{rem:inv} points out, the proof of Theorem~\ref{thm:Invariance} for $\sigma$-atoroidal class $\alpha$ goes through verbatim with only a minor modification. Here we just mention that in order to obtain the corresponding energy estimations we need to replace Lemma~\ref{lem:NAS1} by Lemma~\ref{lem:NAS2} and use the fact that the constants $\epsilon_0$ and $\epsilon_1$ in Lemma~\ref{lem:QII} converge to zero as $|\sigma|_g\to 0$ (see Remark~\ref{rem:lgc}).

\end{proof}

\section{Computations of Floer homology}\label{sec:4}
\setcounter{equation}{0}
In this section we will closely follow the paper by Weber~\cite{We0} to construct two sequences of Hamiltonian functions compactly supported in $D_RT^*M$ and compute their Floer homologies for $T^*M$ endowed  with the canonical symplectic form $\omega_0=-d\lambda$.

\subsection{Radial Hamiltonians}

 Let $H^f:T^*M\to \mathbb{R}$ be an autonomous function of the form
$$H^f(q,p)=f(\|p\|_g)\quad \forall (q,p)\in T^*M,$$
where $f:\mathbb{R}\to\mathbb{R}$ is a smooth function such that $f(r)=f(-r)$. Then the set of the critical points of $\mathscr{A}_{H^f}$ is given by
\begin{eqnarray}\notag
\mathscr{P}_{\alpha}(H^f):=&\big\{& z=(q,p)\in C^\infty(S^1, T^*M)\big|
q(t)\;\hbox{is a geodesic in the class}\;\alpha,\;\notag\\
&&l:=\|\dot{q}\|_g, \;p(t)=\pm\frac{r}{l}\dot{q}(t),\;\hbox{where}\;r>0\; \hbox{satisfies}\; f^\prime(r)=\pm l\big\}.\notag
\end{eqnarray}
Moreover, for each $z\in \mathscr{P}_{\alpha}(H^f)$ it holds that $\mathscr{A}_{H^f}(z)=f^\prime(r)r-f(r)$, that is, the value of the action functional $\mathscr{A}_{H^f}$ at $z$ is equal to the minus $y$-intercept of the tangential line of the graph $y=f(x)$ at $x=r$.

\subsection{The action functional on free loop space}
The action functional $\mathscr{E}$ on $\mathcal{L}_\alpha M$ is defined by
$$\mathscr{E}(q)=\frac{1}{2}\int^1_0\|\dot{x}(t)\|_g^2dt.$$
It is not hard to check that a loop $q\in\mathcal{L}_\alpha M$ is a critical point of $\mathscr{E}$ if and only if $q$ is a $1$-periodic geodesic representing $\alpha$. Given $a\in\mathbb{R}$, denote
$$\mathcal{L}_\alpha^aM:=\{q\in\mathcal{L}_\alpha M\big|\mathscr{E}(q)\leq a\}.$$
For $a\leq b$ the natural inclusion
$$\iota_a^b:\mathcal{L}_\alpha^aM\hookrightarrow \mathcal{L}_\alpha^bM$$
induces the homomorphism
$$[\iota_a^b]:{\rm H}_*(\mathcal{L}_\alpha^aM)\to {\rm H}_*(\mathcal{L}_\alpha^bM).$$
Here ${\rm H}_*(\mathcal{L}_\alpha^aM)$ denotes the singular homology with $\mathbb{Z}_2$-coefficients of the sublevel set $\mathcal{L}_\alpha^aM$.
\begin{remark}\label{rem:nontriviality}
{\rm
The homomorphism $[\iota_a^b]$ is nonzero whenever $l_\alpha\leq a\leq b$.}
\end{remark}

\subsection{Two families of profile functions}\label{subsec:pf}

Fix $c>0$ and pick $a\in(0,c]$ with $a/R\notin\Lambda_\alpha$.
Since the marked length spectrum $\Lambda_\alpha\subseteq\mathbb{R}$ is a closed and nowhere dense subset, there is a dense subset $\Delta$ of $(0,a/c)$ such that for every $\eta\in \Delta$ it holds that
$$\quad \nu_\eta:=\frac{1}{R}\bigg[\frac{a}{\eta}-(c-a)\bigg]\in (a/R,\infty)\setminus\Lambda_\alpha.$$
Fix $\eta\in \Delta$,  using the conventions $\sup\emptyset=0$ and $\inf\emptyset=\infty$, we define
\begin{equation}
\begin{array}{ll}
l_0:=\inf(\Lambda_\alpha\cap(a/ R,\infty)),\quad l_-=l_-(\eta):=\sup((0,\nu_\eta)\cap\Lambda_\alpha),\notag\\
l_+=l_+(\eta):=\inf((\nu_\eta,\infty)\cap\Lambda_\alpha),\quad l_1=l_1(\eta):=\frac{\nu_\eta+l_-}{2}.
\end{array}
\end{equation}
Clearly, we have
$$(a/R,l_0)\cap\Lambda_\alpha=\emptyset,\quad
(l_-,l_+)\cap\Lambda_\alpha=\emptyset\quad \hbox{and}\quad 0\leq l_-<l_1<\nu_\eta<l_+\leq\infty.$$
Denote
\begin{equation}
\begin{array}{lrc}
r_{k1}:=\frac{(k-1)R}{k}+\frac{3}{16k}\big(R-\frac{a}{l_0}\big),\quad
r_{k2}:=R-\frac{3}{16k}\big(R-\frac{a}{l_0}\big),\notag\\
\lambda_k(x):=ak\frac{x-R+(3/(16k))(R-a/l_0)}{R-(3/8)(R-a/l_0)},\qquad k\in\mathbb{N}.
\end{array}
\end{equation}
Here let us remark that $\nu_\eta,l_+(\eta),l_1(\eta)\to+\infty$ as $\eta\to0$, and $r_{k1},r_{k2}\to R$ as $k\to\infty$.

\begin{figure}[H]
  \centering
  \includegraphics[scale=0.5]{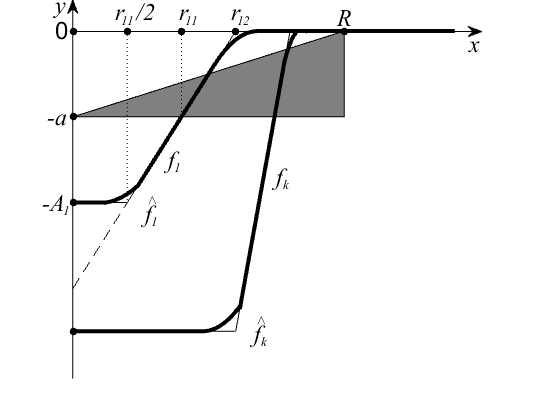}
  \caption{Sequence A of profile functions}\label{fig:1}
\end{figure}


\noindent\textbf{Sequence A of profile functions}: Choose a sequence of smooth functions $f_k$ (see Figure 1) by smoothing out a sequence of piecewise linear functions $\hat{f}_k$, which are given by
\begin{equation}\notag
\hat{f}_k(x):=\left\{
             \begin{array}{ll}
            \lambda_k(r_{k1}/2)&\hbox{if}\;x\in[0,r_{k1}/2),  \\
             \lambda_k(x) & \hbox{if}\;x\in[r_{k1}/2,r_{k2}),\\
             0 & \hbox{if}\;x\in[r_{k2},+\infty).
             \end{array}
\right.
\end{equation}
Each $f_k$ is required to coincide with $\hat{f}_k$ away from sufficiently small neighbourhoods of $r_{k1}/2$ and $r_{k2}$. In particular,  $\hbox{graph}f_k$ equals $\hbox{graph}\hat{f}_k$ in the region that lies below the line $x\to ax/R-a$ and above the line $x\to -a$ (grey region in Figure 1). Besides, we require that $f_k^\prime\geq 0$ everywhere, $f_k^{\prime\prime}\geq 0$ near $r_{k1}/2$, $f_k^{\prime\prime}\leq 0$ near $r_{k2}$, and $f_k^{\prime\prime}=0$ elsewhere. Since $a/R<a/r_{k2}<l_0$, the slope of the unique tangent of the graph of $f_k$ through the point $(0,-a)$ lies in the interval $(a/R,l_0)$. It follows that $a\notin \mathscr{S}_{\alpha}(H^{f_k})$.

\begin{figure}[H]
  \centering
  \includegraphics[scale=0.6]{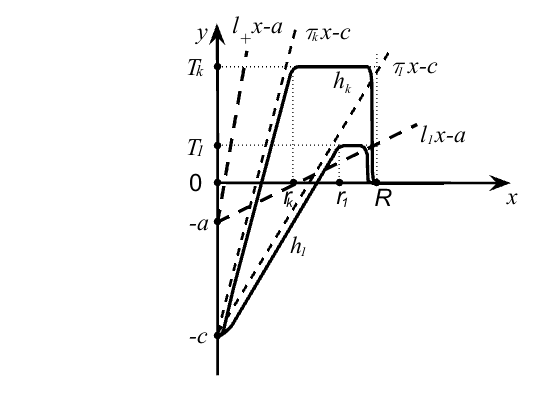}
  \caption{Sequence B of profile functions}\label{fig:2}
\end{figure}


\noindent \textbf{Sequence B of profile functions}: Let $\eta_k\in\Delta$ be a sequence of numbers such that $\eta_k\to 0$ as $k\to\infty$. Denote $$T_k:=\max\{0,l_1(\eta_k)R-a\},\quad \tau_k:=\frac{a}{\eta_kR}
\quad\hbox{and}\quad r_k:=\frac{\eta_kR(T_k+c)}{a}.$$
Note that $T_k\to+\infty$ and $\tau_k\to+\infty$ as $k\to\infty$, and that $0\leq r_k\leq R$. Consider the piecewise linear curve $\Gamma$ in $\mathbb{R}^2$:

\begin{equation}\notag
\forall\; \Gamma\ni(x,y)=\left\{
             \begin{array}{ll}
            \big(x,\tau_kx-c\big) & \hbox{if}\;x\in[0,r_k], \\
             (x,T_k) & \hbox{if}\;x\in[r_k,R],\\
             (R,y) & \hbox{if}\;y\in[0,T_k].
             \end{array}
\right.
\end{equation}
Smoothing out this piecewise linear curve near its corners we obtain a sequence of smooth functions $h_k$ (see Figure 2). Here every $h_k$ is also required to satisfy $h_k^{\prime\prime}\geq 0$ near the point $(0,-c)$ and $h_k^\prime(0)=h_k^\prime(1)=0$. We claim that $a\notin \mathscr{S}_{\alpha}(H^{h_k})$. In fact, the action of $\mathscr{A}_{h_k}$ at some one-periodic orbit is positive if and only if the $y$-intercept of the tangential line of $\hbox{graph}~h_k$ at the  corresponding point is negative. This happens in two clusters. The slope of the tangent of any point in one of those clusters passing through the point $(0,-a)$ lies in the interval $(l_-,l_+)$. The other cluster is located near the point $(0,-c)$, at which the $y$-intercepts of the tangential lines of $\hbox{graph}f_k$ are less than $-c$. So $(l_-,l_+)\cap\Lambda_\alpha=\emptyset$ implies that $a\notin \mathscr{S}_{\alpha}(H^{h_k})$.

\begin{proposition}\label{prop:pf}
Fix $0<a\leq c$ with $a/R\notin\Lambda_\alpha$, and choose $\eta_k\in \Delta$ satisfying $\eta_k\to 0$ as $k\to 0$. Let $\{f_k\}_{k\in\mathbb{N}}$ and $\{h_k\}_{k\in\mathbb{N}}$ be two sequences of those functions constructed above. Choose $k\in\mathbb{N}$ sufficiently large so that
$f_k(0)\leq-c$ and $f_k\leq h_k$. Set $\mu_k:=\nu_{\eta_k}$. Then there exist natural isomorphisms
\begin{equation}\label{e:Ipf1}
\Theta_{f_k}:{\rm HF}^{(a,+\infty)}_{\alpha}(H^{f_k})\to
{\rm H}_*(\mathcal{L}_\alpha^{a^2/(2R^2)}M)
\end{equation}
and
\begin{equation}\label{e:Ipf2}
\Theta_{h_k}:{\rm HF}^{(a,+\infty)}_{\alpha}(H^{h_k})\to
{\rm H}_*(\mathcal{L}_\alpha^{\mu_k^2/2}M)
\end{equation}
such that the following diagram commutes:
\begin{eqnarray}
\begin{CD}\label{diag:dc3}
{\rm HF}^{(a,+\infty)}_{\alpha}(H^{f_k}) @>{\Psi_{h_k f_k}}>> {\rm HF}^{(a,+\infty)}_{\alpha}(H^{h_k})\\
@V{\Theta_{f_k}}V\simeq V  @V\simeq V{\Theta_{h_k}}V \\
{\rm H}_*(\mathcal{L}_\alpha^{a^2/(2R^2)}M) @>{[\iota^{\mu_k^2/2}_{a^2/(2R^2)}}]>> {\rm H}_*(\mathcal{L}_\alpha^{\mu_k^2/2}M)
\end{CD}
\end{eqnarray}

\end{proposition}

The proof of proposition~\ref{prop:pf} is standard, which can be completed by mimicking that of \cite[Theorem~3.1]{We0}. Indeed, in Weber's paper~\cite{We0} he gives the above
result with $R=1$ but by a rescaling argument one gets the result for general $R>0$. 
The basic idea is to deform $f_k$ and $h_k$ by monotone homotopies to convex radial functions so that Theorem~\cite[Theorem~2.9]{We0} can be applied. Full details can be found in Appendix~\ref{app:{prop:pf}}.

\section{Proofs of the main theorems and remarks.}\label{sec:5}
\setcounter{equation}{0}

\begin{figure}[H]
  \centering
  \includegraphics[scale=0.5]{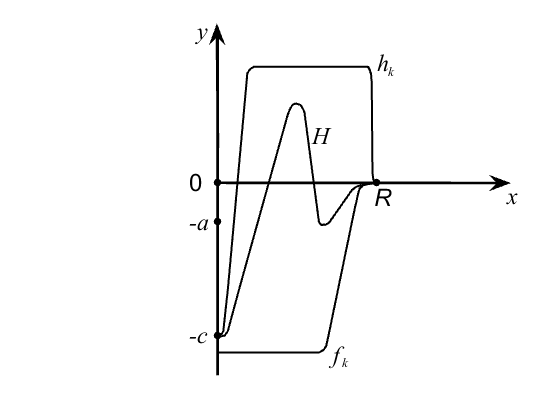}
  \caption{The function $H$}\label{fig:8}
\end{figure}


\subsection{Proof of Theorem~\ref{thm:1}}
\begin{proof}
Let $c=-\sup_{S^1\times M}H$. Then $c> Rl_{\alpha}$.
Choose $a\in[Rl_\alpha,c]$ so that $a/R\notin \Lambda_\alpha$. Then we can construct two sequences of functions $\{f_k\}_{k\in\mathbb{N}}$ and $\{h_k\}_{k\in\mathbb{N}}$ as in Section~\ref{subsec:pf}. For sufficiently large $k$ it follows that
$$H^{f_k}\leq H\leq H^{h_k}$$
as illustrated in Figure~\ref{fig:8}.
Fix such an integer $k=k(H)$. Then by Proposition~\ref{prop:pf} we have the commutative diagram (\ref{diag:dc3}).

So the monotone homomorphism $\Psi_{h_k f_k}$ is nonzero due to our assumption that $\mu_k\geq a/R\geq l_\alpha$; see Remark~\ref{rem:nontriviality}. Applying Theorem~\ref{thm:Invariance} to the Hamiltonians  $H^{f_k}$ and $H^{h_k}$, there exist positive constants $\delta_1=\delta_1(H^{f_k},g,\sigma,a,\alpha)$ and $\delta_2=\delta_2(H^{h_k},g,\sigma,a,\alpha)$ such that if $|\delta|<\delta_0:=\min\{\delta_1, \delta_2\}$, then we have the following commutative diagram

\begin{eqnarray}
\begin{CD}\label{diag:dc5}
{\rm HF}^{(a,+\infty)}_{\alpha}(H^{f_k}) @>{\Psi_{h_kf_k}}>> {\rm HF}^{(a,+\infty)}_{\alpha}(H^{h_k})\\
@V{\Psi_{\omega_0}^{\omega_{\delta\sigma}}}VV  @VV{\Psi_{\omega_0}^{\omega_{\delta\sigma}}}V \\
 {\rm HF}^{(a,+\infty)}_{\alpha}(H^{f_k},\delta\sigma) @>{\Psi^{\delta\sigma}_{h_kf_k}}>> {\rm HF}^{(a,+\infty)}_{\alpha}(H^{h_k},\delta\sigma)
\end{CD}
\end{eqnarray}
Combining (\ref{diag:dc3}) and (\ref{diag:dc5}) we deduce that  for $|\delta|<\delta_0$ the monotone homomorphism
$\Psi^{\delta\sigma}_{h_kf_k}$ is also nonzero. Given $\delta\in(-\delta_0,\delta_0)$,  choose a sequence of Hamiltonian functions $H_i \in \mathscr{H}$ such that $H_i$ satisfy the non-degeneracy condition (C) and converge to $H$ in $C^\infty$ topology, $H^{f_k}\leq H_i\leq H^{h_k}$, $a\notin\mathscr{S}_\alpha(H_i,\delta\sigma)$ and $\sup_{S^1\times M}H_i<- Rl_{\alpha}$. Then, by Lemma~\ref{lem:mh}, the non-trivial  homomorphism
$\Psi^{\delta\sigma}_{h_kf_k}$ factors through the Floer homology group ${\rm HF}^{(a,+\infty)}_{\alpha}(H_i,\delta\sigma)$. Hence there exists a sequence of periodic orbits $x_i\in\mathscr{P}_{\alpha}(H_i,\delta\sigma)$ such that
$\mathscr{A}_{H_i,\delta\sigma}(x_i)> a$. Passing to a converging subsequence, we get a periodic orbit $x\in\mathscr{P}_{\alpha}(H,\delta\sigma)$ with $\mathscr{A}_{H,\delta\sigma}(x)\geq a$.

\end{proof}


\subsection{Proof of Theorem~\ref{thm:2}}

\begin{proof}
Consider the Hamiltonian function defined by
$$\bar{H}(t,x):=-H(-t,x)\quad \forall\;(t,x)\in S^1\times D_RT^*M.$$
Let $C(\alpha)=Rl_\alpha$. Obviously, $x(t)$ is a periodic orbit of $H$ representing $-\alpha$ if and only if $x(-t)$ is a periodic orbit of $\bar{H}$ representing $\alpha$, and
it holds that
$$\sup_{S^1\times M}\bar{H}\leq -C(\alpha).$$
Imitating the proof of Theorem~\ref{thm:1} with replacing Theorem~\ref{thm:Invariance} by Theorem~\ref{thm:Inv} concludes the proof of Theorem~\ref{thm:2}.

\end{proof}

\subsection{Proofs of Theorem~\ref{thm:3} and Theorem~\ref{thm:4}}
The idea of the proof of Theorem~\ref{thm:3} is the same as that of Theorem~\ref{thm:1}, while the difference between the two proofs is that one can squeeze uniformly a class of functions (whose graphs lie in the grey region in Figure~\ref{fig:9} and on the line $y=0,x\geq R-\rho$) from above and below in the proof of Theorem~\ref{thm:3}. The proof of Theorem~\ref{thm:4} is similar to that of Theorem~\ref{thm:3}. Here we only show the latter. 
\begin{figure}[H]
  \centering
  \includegraphics[scale=0.5]{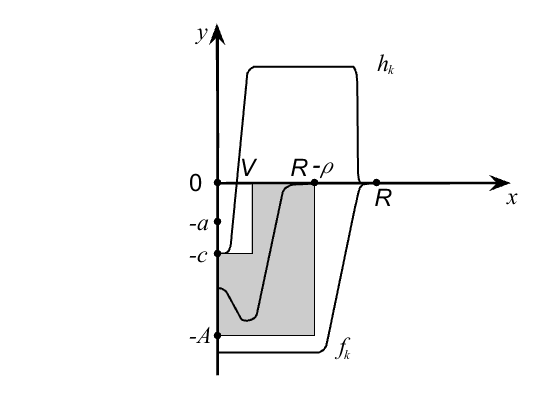}
  \caption{A class of functions}\label{fig:9}
\end{figure}


\begin{proof}
Let $\epsilon>0$ be an arbitrary number in $(0,A-Rl_\alpha)$ and set $c=Rl_\alpha+\epsilon$. Pick $a=a(\epsilon)\in[Rl_\alpha,c]$ so that $a/R\notin \Lambda_\alpha$. Construct two sequences of functions $\{f_k\}_{k\in\mathbb{N}}$ and $\{h_k\}_{k\in\mathbb{N}}$ as in Section~\ref{subsec:pf}. Then there exists an integer $k_0=k_0(g,\sigma,\alpha, V,A,\rho,\epsilon)>0$ such that
$$H^{f_{k_0}}\leq H\leq H^{h_{k_0}}$$
for every $H\in\mathscr{H}_c(U_{R},U_{R-\rho},V,A)$.
Following the proof of Theorem~\ref{thm:1} shows that
there exists a constant
$$\delta_0:=\delta_0(g,\sigma,\alpha, V,A,\rho,\epsilon)>0$$ such that if $|\delta|<\delta_0$ then every $H\in\mathscr{H}_c(U_{R},U_{R-\rho},V,A)$  admits a $1$-periodic Hamiltonian orbit with respect to the twisted symplectic form $\omega_{\delta\sigma}$ whose projection to $M$ represents $\alpha$. Therefore, by definition, for each $\delta\in(-\delta_0,\delta_0)$ we have $\hat{c}_{\rm BPS}(U_R,U_{R-\rho},V,A;\delta\sigma,\alpha)\leq Rl_\alpha+\epsilon$.

\end{proof}

\subsection{Proofs of Theorem~\ref{thm:AET1} and Theorem~\ref{thm:AET2}}
The key to the proofs of Theorem~\ref{thm:AET1} and Theorem~\ref{thm:AET2} is the following monotonicity of the restricted BPS capacity.
\begin{proposition}\label{prop:monotonicity}
Let $W_1,W_2$ be open subsets of $T^*M$ containing $M$, $V_2\subset V_1\subset U_1\subset U_2$, $W_1\subset W_2$ and $0<A_1\leq A_2$, where $W_i,U_i,V_i$, $i=1,2$ are open subsets of $T^*M$ and $U_i$ have compact closure $\bar{U}_i\subset W_i$. Then it holds that
$$\hat{c}_{\rm BPS}(W_1,U_1,V_1,A_1;\sigma,\alpha)\leq\hat{c}_{\rm BPS}(W_2,U_2,V_2,A_2;\sigma,\alpha).$$
\end{proposition}

Our proof of Theorem~\ref{thm:AET1} is an adaption of the standard almost existence theorem (see~\cite[Section~4.2]{HZ} or~\cite{SW}).
For the sake of completeness we shall give a proof of Theorem~\ref{thm:AET1}. The proof of Theorem~\ref{thm:AET2} is nearly identical to that of Theorem~\ref{thm:AET1}, we omit it here.
\begin{proof}
Along the lines of~\cite[Section~4.2]{HZ}, we proceed in 3 steps.

\noindent \textbf{Step 1.}
By our assumption, the sublevel $\{H<s\}$ is contained in $U_{r(s)}$, where $r:\mathbb{R}\to (0,\infty)$ is a nondecreasing function. Consider the monotone functions $\hat{c}_{\delta}:[d,T]\to[0,\infty]$ defined by
$$\hat{c}_{\delta}(s):=\hat{c}_{\rm BPS}(U_{r(T+\rho)+\rho},\{H<s\},V,A;\delta\sigma,\alpha).$$
Theorem~\ref{thm:3} implies that there exists a constant
$\delta_0:=\delta_0(g,\sigma,\alpha, V,A,T,\rho)>0$ such that if $|\delta|<\delta_0$ then
$$\hat{c}_{\rm BPS}(U_{r(T+\rho)+\rho},U_{r(T+\rho)},V,A;\delta\sigma,\alpha)< A$$
provided $A>(r(T+\rho)+\rho)l_\alpha$.
Then we deduce from Proposition~\ref{prop:monotonicity} that
$$\hat{c}_{\delta}(s)\leq \hat{c}_{\delta}(T+\rho)\leq \hat{c}_{\rm BPS}(U_{r(T+\rho)+\rho},U_{r(T+\rho)},V,A;\delta\sigma,\alpha)< A\quad \forall s\in[d,T+\rho],$$
where $|\delta|<\delta_0(g,\sigma,\alpha, V,A,T,\rho)$. Fix $\delta\in(-\delta_0,\delta_0)$.
So Lebesgue's last theorem implies that the function $\hat{c}_{\delta}$ is
differentiable at almost every point in the sense of measure theory. Suppose that $s_0\in [d,T]$ is a regular value of $H$ and $\hat{c}_{\delta}$ is Lipschitz continuous at $s_0$. Then, by the
implicit function theorem, $S_0:=H^{-1}(s_0)$ is a hypersurface in $T^*M$,  which is compact and bounds
the sublevel set $B_0:=\{H<s_0\}$ since $H$ is proper and bounded below. By the implicit function theorem again and compactness of
$S_0$, one can find a parameterized family of hypersurfaces $S_\varepsilon$ in $T^*M$ with $S_\varepsilon=H^{-1}(s_0+\varepsilon)$ such that
$$c(\varepsilon)\leq c(0)+L\varepsilon,\quad c(\varepsilon):=\hat{c}_{\delta}(s_0+\varepsilon)$$
for $0\leq\varepsilon\leq\eta$, where $L$ and $\eta$ are positive constants. By choosing $\eta$ smaller, let us require that $\eta<\min\{\rho,(A-c(0))/(2L)\}$ ($A>c(0)$ by our assumption). For  $0<\varepsilon\leq \eta$, let us denote $B_\varepsilon:=\{H<s_0+\varepsilon\}$, then
$B_\varepsilon\subset \{H<T+\rho\}\subset U_{r(T+\rho)}$.
Fixing $\varsigma\in(0,\eta]$, we define a smooth function $f:\mathbb{R}\to [-2L\varsigma,0]$ by
\begin{equation}\notag
             \begin{array}{ll}
            f(s)=-2L\varsigma & \hbox{if}\;s\leq 0 \\
            f(s)=0 & \hbox{if}\;s\geq\frac{\varsigma}{2}\\
            0<f^\prime(s)\leq 8L & \hbox{if}\;0<s<\frac{\varsigma}{2}.
             \end{array}
\end{equation}
By the definition of the restricted BPS capacity $\hat{c}_\delta(s_0)=c(0)$, there exists a Hamiltonian function $G\in\mathscr{H}(U_{r(T+\rho)+\rho},B_0,A)$ such that
\begin{equation}\label{e:funG}
-c(0)<\sup_{V}G\leq -(c(0)-L\varsigma)
\quad \hbox{and}\quad \mathscr{P}_{\alpha}(G,\delta\sigma;\tau)=\emptyset\quad\forall\; 0<\tau\leq1.
\end{equation}
Otherwise, $c(0)=\hat{c}_\delta(s_0)\leq c(0)-L\varsigma$ which contradicts with the fact that $L\varsigma\geq 0$.
Now we constructed a new function $\widetilde{G}$ by cutting off $G$ without creating new fast periodic orbits.
More precisely, let $\chi:[-c(0)-\epsilon,-c(0)+\epsilon]\to \mathbb{R}$ be a function with $0\leq \chi^\prime\leq 1$ such that $\chi(t)=-c(0)$ for $t$ near the left endpoint of the interval and $\chi(t)=t$ for $t$ near the right endpoint of the interval, where $\epsilon>0$ is sufficiently small
so that $-c(0)+\epsilon<\sup_{V}G$. Set
\begin{equation}\notag
\widetilde{G}(x):=\left\{
             \begin{array}{ll}
            -c(0)&\hbox{if}\;G(x)\leq -c(0)-\epsilon,  \\
             \chi(G(x)) & \hbox{if}\;-c(0)-\epsilon\leq G(x)\leq-c(0)+\epsilon,\\
             G(x) & \hbox{if}\;-c(0)+\epsilon\leq G(x).
             \end{array}
\right.
\end{equation}
The new non-positive function $\widetilde{G}$ compactly supported in $B_0$ satisfies $$\mathscr{P}_{\alpha}(\widetilde{G},\delta\sigma;\tau)
=\emptyset\quad\forall\; 0<\tau\leq1$$ since $|\chi^\prime|\leq 1$. Furthermore, $\sup_{V}\widetilde{G}\leq -(c(0)-L\varsigma)$ and $\inf_{B_0}\widetilde{G}\geq -c(0)$.
Consider the compactly supported Hamiltonian function $K\in C^\infty(B_\varsigma)$ defined by
\begin{equation}\notag
             \begin{array}{ll}
            K(x)=\widetilde{G}(x)-2L\varsigma & \hbox{if}\;x\in B_0 \\
            K(x)=f(\varepsilon) & \hbox{if}\;x\in S_\varepsilon,\; 0\leq \varepsilon<\varsigma\\
            K(x)=0 & \hbox{if}\;x\notin B_\varsigma.
             \end{array}
\end{equation}
The function $K$ satisfies
$$\sup_{V}K=\sup_{V}\widetilde{G}-2L\varsigma \leq -(c(0)-L\varsigma)-2L\varsigma\leq -c(\varsigma)\quad\hbox{and}$$
$$\inf_{B_\varsigma}K= \inf_{B_0}\widetilde{G}-2L\varsigma
> -c(0)-2L\eta> -A.$$
The definition of the restricted BPS capacity $c(\varsigma)$ shows that
$K$ has a fast periodic orbit $x$ with respect to $\omega_{\delta\sigma}$ whose projection to $M$ represents $\alpha$. We claim that $x$ cannot intersect $B_0$. Indeed, if $x$ intersects $B_0$ then it stays completely inside $B_0$ since $B_0$ is invariant under the flow of $K$. This is impossible because the flows of $K$ and $\widetilde{G}$ on $B_0$ coincide and
$\widetilde{G}$ does not have fast periodic orbits. Since $\alpha\neq 0$, the fast periodic orbit $x$ whose projection represents $\alpha$ is nontrivial.
As a consequence, $x$ must be contained in $B_\varsigma\setminus \overline{B_0}$, and hence it lies on
$S_\varepsilon$ for some $0< \varepsilon<\varsigma$.

\noindent \textbf{Step 2.} Step 1 works for every $\varsigma\in(0,\eta]$. Choosing a sequence $\varsigma_j\to 0$, one can find sequences $K_j$ and $\varepsilon_j$, and a corresponding sequence $x_j(t)$ of periodic orbits of $X_{K_j,\delta\sigma}$ having periods $0<\tau_j\leq 1$ and lying on
$S_{\varepsilon_j}$ with $\varepsilon_j\to 0$. Consider the Hamiltonian $H$ on the set
$$U=\bigcup\limits_{\varepsilon\in(-\eta,\eta)}S_\varepsilon.$$
Obviously,  if $x\in S_\varepsilon$ then $H(x)=s_0+\varepsilon$ and $K_j(x)=f_j(\varepsilon)=f_j(H(x)-s_0)$. By construction, the periodic orbits $x_j$ solve the equations
\begin{equation}\notag
             \begin{array}{ll}
            \dot{x}_j(t)=f^\prime_j(\varepsilon_j)X_{H,\delta\sigma}\big(x_j(t)\big)\\
            x_j(0)=x(\tau_j)
             \end{array}
\end{equation}
with the periods $0\leq\tau_j\leq1$. Normalizing the periods to $1$ we define the functions
$$y_j(t)=x_j\big(\tau_jt\big)\quad \forall t\in[0,1]$$
which solve the Hamiltonian equations
$$\dot{y}_j(t)=f^\prime_j(\varepsilon_j)\tau_jX_{H,\delta\sigma}(y_j(t))\quad \hbox{and}\quad H\big(y_j(t)\big)=\varepsilon_j.$$

\noindent \textbf{Step 3.} By construction, $f^\prime_j\leq 8L$ and hence $f^\prime_j(\varepsilon_j)\tau_j$ are bounded. This observation is very useful for us to obtain a periodic orbit on $S_0$. Indeed, we first note that the hypersurfaces $S_{\varepsilon_j}$ are contained in the compact set $\overline{B_\eta}$, hence the functions $x_j$ are uniformly bounded. For all $t\in S^1$ and all $j\in \mathbb{N}$ we estimate
$$\|\dot{y}_j(t)\|_{G_g}=|f^\prime_j(\varepsilon_j)\tau_j|\cdot \|X_{H,\delta\sigma}\big(x_j(t)\big)\|_{G_g}\leq 8L \sup_{x\in \overline{B_\eta}}\|X_{H,\delta\sigma}\|_{G_g}.
$$
Then, by Arzela-Ascoli theorem, passing to subsequences, $f^\prime_j(\varepsilon_j)\tau_j$ converges to some $\tau\geq0$ and $y_j(t)$ converges in $C^0$-topology, and,  by making use of the equations, even converges in $C^\infty$-topology to a smooth $1$-periodic solution $y$ of the equation
$$\dot{y}(t)=\tau X_{H,\delta\sigma}(y(t)),\quad y(t)\subset S_0.$$
We claim that $\tau\neq 0$. Otherwise, $y(t)=y^*$ for some point $y^*\in S_0$. This contradicts with the fact that
the projection of $y$ on $M$ represents the non-trivial class $\alpha$ since $[\pi (y_j)]=\alpha$ and $\pi (y_j)$ converges to $ \pi(y)$ in $C^\infty$-topology. Reparametrizing time we obtain
the $\tau$-periodic solution $x(t):=y(t/\tau)$ of the equation
$$\dot{x}(t)=X_{H,\delta\sigma}(x(t))$$
whose projection on $M$ represents $\alpha$. The proof of Theorem~\ref{thm:AET1} is completed.
\end{proof}

\subsection{Concluding remarks}

\subsubsection{Hamiltonian flows without non-contractible closed trajectories}
The following proposition shows that the constant $\delta_0(H,g,\sigma,a,b,\alpha)$ in Theorem~\ref{thm:1} depends on $H$. As a consequence, one can not extend Theorem~\ref{thm:1} (resp. Theorem~\ref{thm:2}) to the case that $\delta_0(H,g,\sigma,a,b,\alpha)$ (resp. $\delta_0(H,g,a,b,\alpha)$) is arbitrarily large.
\begin{proposition}~\label{pro:nonexistence}
Let $M$ be a closed Riemanian surface endowed with a metric $g$ of constant curvature $K=-1$. Then there exists a sequence of compactly supported Hamiltonians $\{H_n\}_{n\in \mathbb{N}}\subset C^\infty(D_1T^*M)$ with $\inf_{M}H_n>n$ and a sequence of numbers $\{\delta_n\}_{n\in \mathbb{N}}$ converging to $0$ such that the periodic orbits of the Hamiltonian flow of $H_n$ with respect to $\omega_n=\omega_0+\delta_n\pi^*\sigma$ are all contractible.

\end{proposition}
This proposition is due to Niche~\cite{Ni}. For the sake of completeness, we outline a proof of Proposition~\ref{pro:nonexistence} below. The following proof is different from Niche's proof and is based on a theorem of Ginzburg~\cite{Gi1}.

\begin{proof}
We start with $\omega=\omega_0+\pi^*dA$ where $dA$ is the area form on $(M,g)$. Further, let $F$ be the standard kinetic energy Hamiltonian.
Then \cite[Theorem~2.5]{Gi1} shows that on every level $\{F=c\}$ with $c<1/2$ all integral curves of $F$ are closed and contractible, and there is no closed orbit on $\{F=1/2\}$  (on which the Hamiltonian flow is the horocycle flow, cf.~\cite{He}). Let $V=\{F<1/2\}$. Let $\chi:[0,1/2]\to [0,C]$ be a ``one-sided bump" function with $\chi^\prime\leq 0$ such that $\chi(t)=C$  near $t=0$ and $\chi(t)=0$ near $t=1/2$. Here $C$ is a constant and can be made arbitrarily large. Let $H=\chi\circ F$. Then $H$ is supported in $V$. Moreover, the flow of $H$ is essentially a reparametrization of the flow of $H$ on $V$ and hence all orbits are closed and contractible.

Let us now replace the area form $dA$ by a new magnetic field $\delta dA$ where $0<\delta \leq 1$. The region $V$ will shrink to $V_\delta $ or in other words the threshold level $\{F=1/2\}$ will get closer to $M$ getting replaced by $\{F=\delta^2 /2\}$. But everything else remains the same. Indeed, by a rescaling argument, the existence of closed trajectories of $\omega=\omega_0+\delta \pi^*dA$ on the energy level $\{F=\delta^2 /2\}$ is equivalent to the existence of closed trajectories of $\omega=\delta \omega_0+\delta \pi^*dA$ on the energy level $\{F=1/2\}$. The flow of the later is a reparametrization of the flow of $F$ with respect to $\omega=\omega_0+\pi^*dA$ on $\{F=1/2\}$, and hence has no non-contractible periodic obits.

Finally, by taking some sequence $\{\delta _n\}_{n=1}^{\infty}$ satisfying $\delta _n\to 0$ and a sequence of Hamiltonians $H_n$ obtained by composing the same $F$ with more and more narrow bump functions, we obtain the desired result.

\end{proof}

\subsubsection{Counterexample}
The condition $\{H<d\}\supset M$ in Theorem~\ref{thm:AET1} and Theorem~\ref{thm:AET2} cannot be dropped. The following example is given by Salom\~{a}o and Weber~\cite{SW}.
\begin{example}
{\rm
Let $M=S^1=\mathbb{R}/\mathbb{Z}$. Consider the function $H:T^*M=S^1\times \mathbb{R}\to \mathbb{R}$ given by
$$H(q,p)=\frac{1}{2}\|p\|_g^2+V(q),$$
where $V(q)=1+cos2\pi q$. Then $\{H<1\}$ does not contain $M$, and for any energy $s\in[1, 2)$ the level set ${H=s}$ does not carry non-contractible periodic orbits.
\rm}
\end{example}

\section{Appendix. The proof of Proposition~\ref{prop:pf}}\label{app:{prop:pf}}
To prove Proposition~\ref{prop:pf}, we need a theorem of Weber in~\cite{We0} that computes Floer homology of convex radial Hamiltonians.

\begin{theorem}[{\cite[Theorem~2.9]{We0}}]\label{thm:crH}
	Let $f:\mathbb{R}\to\mathbb{R}$ be a smooth symmetric function with $f^{\prime\prime}\geq0$. If $\tau\in\mathbb{R}^+\setminus\Lambda_\alpha$ and $f^{\prime}(r)=\tau$ for some $r>0$, then there is a natural isomorphism
	\begin{equation}\label{e:B}
	\Phi_f^\tau:{\rm HF}^{(-\infty,b_{f,\tau})}_{\alpha}(H^f)\to
	{\rm H}_*(\mathcal{L}_\alpha^{\tau^2/2}M),\quad b_{f,\tau}:=rf^{\prime}(r)-f(r).
	\end{equation}
	If $H^h$ is another such Hamiltonian, then there exists an isomorphism $\Psi_{hf}^\tau$ such that the following diagram commutes:
	\begin{eqnarray}\label{e:crHdc}
	\xymatrix{{\rm HF}^{(-\infty,b_{f,\tau})}_{\alpha}(H^f)
		\ar[rr]_{\simeq}^{\Psi_{hf}^\tau}\ar[dr]^{\simeq}_{\Phi_f^\tau}& &{\rm HF}^{(-\infty,b_{h,\tau})}_{\alpha}(H^h)\ar[dl]_{\simeq}^{\Phi_h^\tau}\\ & {\rm H}_*(\mathcal{L}_\alpha^{\tau^2/2}M)&  }
	\end{eqnarray}
	If $\rho\in(0,\tau]\setminus\Lambda_\alpha$ and $f^{\prime}(s)=\rho$ for some $s>0$, then we have the following communicative diagram whose top horizonal row is the natural inclusion $\iota^F$:
	\begin{eqnarray}
	\begin{CD}\label{diag:dc2}
	{\rm HF}^{(-\infty,b_{f,\rho})}_{\alpha}(H^f) @>{\iota^F}>> {\rm HF}^{(-\infty,b_{f,\tau})}_{\alpha}(H^f)\\
	@V{\Phi_{f}^\rho}V\simeq V  @V\simeq V{\Phi_{f}^\tau}V \\
	{\rm H}_*(\mathcal{L}_\alpha^{\rho^2/2}M) @>{[\iota^{\tau^2/2}{\rho^2/2}}]>> {\rm H}_*(\mathcal{L}_\alpha^{\tau^2/2}M)
	\end{CD}
	\end{eqnarray}
	
\end{theorem}
\noindent Here the Floer homology ${\rm HF}^{(a,b)}_{\alpha}(H^f)$ is well defined for every $H^f\in \mathscr{K}_{R;\alpha}^{a,b}$; see Remark~\ref{rem:HFfcsf}.


\begin{figure}[H]
	\centering
	\includegraphics[scale=0.5]{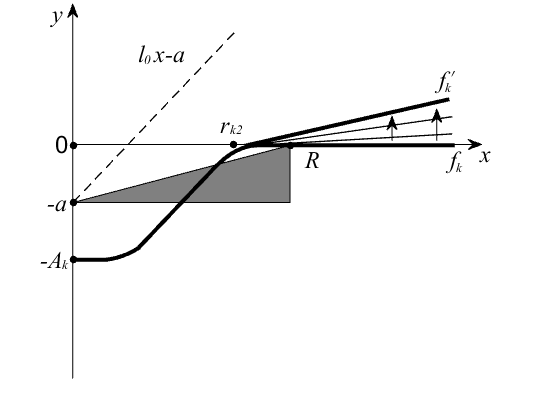}
	\caption{The monotone homotopy between $f_k$ and $f_k^\prime$}\label{fig:3}
\end{figure}


\begin{proof}[Proof of Proposition~\ref{prop:pf}]
	The basic idea is to deform $f_k$ and $h_k$ by monotone homotopies to convex functions so that Theorem~\ref{thm:crH} can be applied. To show the isomorphism~(\ref{e:Ipf1}), first, we follow the graph of $f_k$ until the slope of it becomes $a$ for the second time at a point, saying $p$, near $r_{k2}$, then continue linearly with slope $a$, we obtain a function which is of $C^1$-class at $p$.
	Then smoothing out such a function near $p$ yields a smooth function  $f_k^\prime\in \mathscr{K}_{R;\alpha}^{-\infty,a}$ (see Figure~\ref{fig:3}). The monotone homotopy between $f_k$ and $f_k^\prime$, as showed in Figure~\ref{fig:3}, provides the monotone isomorphism
	\begin{equation}\label{e:hmfk1}
	\Psi_{f_k^\prime f_k}:{\rm HF}^{(a,+\infty)}_{\alpha}(H^{f_k})\to
	{\rm HF}^{(a,+\infty)}_{\alpha}(H^{f^\prime_k}).
	\end{equation}
	This is the consequence of the fact that the $y$-intercepts of the tangential line at all points that do not remain constant during the homotopy are strictly larger than $-a$.
	
	\begin{figure}[H]
		\centering
		\includegraphics[scale=0.5]{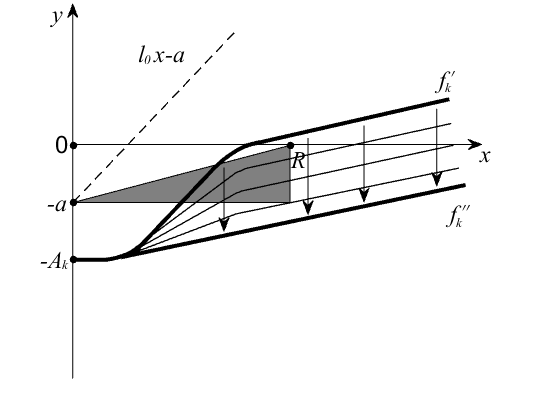}
		\caption{The monotone homotopy between $f_k^\prime$ and $f_k^{\prime\prime}$}\label{fig:4}
	\end{figure}
	
	Second, we define the function $f_k^{\prime\prime}$ obtained by following the the graph $f_k^\prime$ until the slope of it becomes $a$ for the first time at some point $q$; then continue linearly with slope $a$ (see Figure~\ref{fig:4}). The result function is of $C^1$-class at $q$. Smoothing near $q$ we obtain a smooth function $f_k^{\prime\prime}\in\mathscr{K}_{R;\alpha}^{-\infty,a}$.
	Note that all tangential lines of the graphs during the monotone homotopy, as indicated in Figure~\ref{fig:4}, which pass through the point $(0,-a)$ lie strictly between the lines $x\mapsto ax/R-a$ and $x\mapsto l_0x-a$. It follows that the homotopy between $f_k^\prime$ and $f_k^{\prime\prime}$ induces the monotone isomorphism
	\begin{equation}\label{e:hmfk2}
	\Psi_{f_k^{\prime\prime}f_k^\prime}:{\rm HF}^{(a,+\infty)}_{\alpha}(H^{f^\prime_k})\to
	{\rm HF}^{(a,+\infty)}_{\alpha}(H^{f_k^{\prime\prime}}).
	\end{equation}
	Since the $y$-intercept of the tangential line of graph $f_k^{\prime\prime}$ at any point is less than $-a$, the action of
	$\mathscr{A}_{H^{f_k^{\prime\prime}}}$ at each $1$-periodic orbit
	is larger than $a$. Then the exact sequence~(\ref{e:esFH}) yields isomorphisms
	\begin{equation}\label{e:hmfk3}
	{\rm HF}^{(a,+\infty)}_{\alpha}(H^{f_k^{\prime\prime}})\xrightarrow
	{}
	{\rm HF}^{(-\infty,+\infty)}_{\alpha}(H^{f_k^{\prime\prime}})
	\to {\rm HF}^{(-\infty,B)}_{\alpha}(H^{f_k^{\prime\prime}}).
	\end{equation}
	Here $B=b_{f_k^{\prime\prime},a/R}$ is defined in Theorem~\ref{thm:crH} (the fact that $a/R\notin\Lambda_\alpha$ is used). Then by Theorem~\ref{thm:crH} we have the natural isomorphism
	\begin{equation}\label{e:hmfk4}
	\Phi_{f_k^{\prime\prime}}^{a/R}:{\rm HF}^{(-\infty,B)}_{\alpha}(H^{f_k^{\prime\prime}})\to
	{\rm H}_*(\mathcal{L}_\alpha^{a^2/(2R^2)}M).
	\end{equation}
	By composing~(\ref{e:hmfk1})-(\ref{e:hmfk4}) we arrive at the desired isomorphism~(\ref{e:Ipf1}).
	
	\begin{figure}[H]
		\centering
		\subfigure[The monotone homotopy between $h_k$ and $h_k^{\prime}$]{
			\label{fig:6a}
			\includegraphics[scale=0.38]{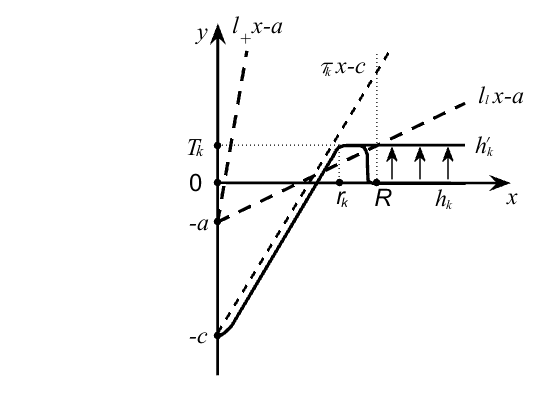}}
		\subfigure[The monotone homotopy between $h_k^\prime$ and $h_k^{\prime\prime}$]{
			\label{fig:6b}
			\includegraphics[scale=0.38]{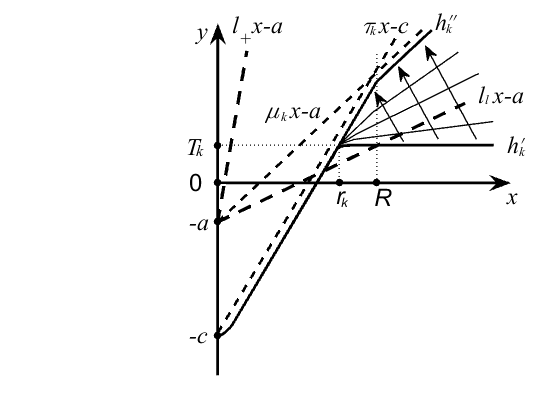}}
		\caption{Monotone homotopies}
		\label{fig:6}
	\end{figure}
	
	To obtain the isomorphism~(\ref{e:Ipf2}), we deform $h_k$ initially by the monotone homotopy as showed in Figure~\ref{fig:6}~(a) to a smooth function $h_k^\prime$, then keep on deforming $h_k^\prime$ by the monotone homotopy (see Figure~\ref{fig:6}~(b)) to a smooth function $h_k^{\prime\prime}$. Here if $x\geq R+\varepsilon$ for some sufficiently small positive constant $\varepsilon$, the graph of $h_k^{\prime\prime}$ turns into a ray with slope $\mu_k$ which is very close to the line $x\to \mu_kx-a$. Note that all points on members of the homotopy whose tangential lines pass through the point $(0,-a)$ lie strictly between the lines $l_-x-a$ and $l_+-a$ ; see Figure~\ref{fig:6}. Therefore we obtain the monotone isomorphisms
	\begin{equation}\label{e:hmhk1}
	{\rm HF}^{(a,+\infty)}_{\alpha}(H^{h_k})\xrightarrow{\Psi_{h_k^{\prime} h_k}} {\rm HF}^{(a,+\infty)}_{\alpha}(H^{h^{\prime}_k})
	\xrightarrow{\Psi_{h_k^{\prime\prime} h_k^{\prime} }}
	{\rm HF}^{(a,+\infty)}_{\alpha}(H^{h^{\prime\prime}_k}).
	\end{equation}
	
	\begin{figure}[H]
		\centering
		\includegraphics[scale=0.5]{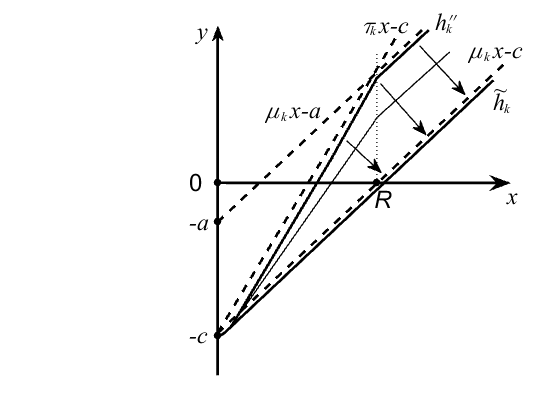}
		\caption{The monotone homotopy between $h_k^{\prime\prime}$ and $\tilde{h}_k$}\label{fig:7}
	\end{figure}
	
	Next, consider function $\tilde{h}_k$ given by following the graph of $h_k^{\prime\prime}$ until it takes on slope $\mu_k$ for the first time (near the point $(0,-c)$); then continue linearly with slope $\mu_k$ (see~Figure~\ref{fig:7}). Smoothing near $(0,-c)$ yields $\tilde{h}_k$. Since the $y$-intercepts of the tangential lines
	of the graphs during the homotopy as showed in~\ref{fig:7} are less than $-a$, we obtain the monotone isomorphism
	\begin{equation}\label{e:hmhk2}
	\Psi_{\tilde{h}_kh_k^{\prime\prime}}:{\rm HF}^{(a,+\infty)}_{\alpha}(H^{h_k^{\prime\prime}})\to
	{\rm HF}^{(a,+\infty)}_{\alpha}(H^{\tilde{h}_k}).
	\end{equation}
	Again, using the fact that the $y$-intercepts of the tangential lines of the graph $\tilde{h}_k$ are less than $-a$,
	the exact sequence~(\ref{e:esFH}) implies the following isomorphisms:
	\begin{equation}\label{e:hmhk3}
	{\rm HF}^{(a,+\infty)}_{\alpha}(H^{\tilde{h}_k})\to {\rm HF}^{(-\infty,+\infty)}_{\alpha}(H^{\tilde{h}_k})
	\to{\rm HF}^{(-\infty,B^\prime)}_{\alpha}(H^{\tilde{h}_k})
	\end{equation}
	where $B^\prime=b_{\tilde{h}_k,\mu_k}$ is as in~(\ref{e:B}). Then Theorem~\ref{thm:crH} gives the isomorphsim
	\begin{equation}\label{e:hmhk4}
	\Phi_{\tilde{h}_k}^{\mu_k}:{\rm HF}^{(-\infty,B^\prime)}_{\alpha}(H^{\tilde{h}_k})\to
	{\rm H}_*(\mathcal{L}_\alpha^{\mu_k^2/2}M).
	\end{equation}
	Composing ~(\ref{e:hmhk1}) - (\ref{e:hmhk4}) yields the desired isomorphism (\ref{e:Ipf2}).
	
	The proof of (\ref{diag:dc3})
	is identical to that of the communicative diagram~(53) in~\cite{We0} without any essential changes; see p.563 and p. 564, which is an easy application of Theorem~\ref{thm:crH}, we omit it here.
\end{proof}



\end{document}